\documentclass{amsart}
\usepackage{amsmath,amssymb,amsthm,xcolor,enumerate}
\usepackage[mathscr]{euscript}
\makeatletter
\@namedef{subjclassname@2020}{%
  \textup{2020} Mathematics Subject Classification}
\makeatother

\newcommand{\p}{p_0}

\def\ds{\displaystyle}
\usepackage[normalem]{ulem}
\def \R {{\mathbb R}}
\def \N {{\mathbb N}}
\def \Z {{\mathbb Z}}
\def \C {\mathcal{C}}
\def \d {\mathbf d}
\def \U {\mathcal U}

\def\L{\mathcal{L}}
\def\K{\mathcal{K}}
\def\S{\mathcal{S}}
\def\T{\mathcal{C}}
\def\M{\mathcal{M}}

\def \vsm{\vskip 0.2 truecm}

\def \ol {\overline}

\newtheorem{theorem}{Theorem}[section]
\newtheorem{lemma}{Lemma}[section]
\newtheorem{proposition}{Proposition}[section]
\theoremstyle{definition}
\newtheorem{definition}[theorem]{Definition} 
\theoremstyle{remark}
\newtheorem{remark}[theorem]{Remark}

\begin{document}

\title[A Converse Lyapunov theorem with regulated cost]{A converse Lyapunov-type theorem  \\for control systems with regulated cost} 
\thanks{*Corresponding author}
\thanks{This research is partially supported by the  INdAM-GNAMPA Project 2023 CUP E53C22001930001.}

\author[Anna Chiara Lai]{Anna Chiara Lai}
\address{Anna Chiara Lai. Dipartimento di Scienze di Base e Applicate per l'Ingegneria, Sapienza Università di Roma, via Antonio Scarpa,10 - 00161, Rome, Italy} 
\email{annachiara.lai@uniroma1.it}

\author[Monica Motta]{Monica Motta* }
\address{Monica Motta. Dipartimento di Matematica ``Tullio Levi-Civita", Università di Padova, via Trieste, 63 - 35121, Padua, Italy}
\email{motta@math.unipd.it}

\keywords
{ Converse Lyapunov-type theorem , Asymptotic controllability with regulated cost,  Optimal control , Nonlinear theory, Viscosity solutions}
\subjclass[2020]{ 93D30, 93B05,  49J15, 93C10,  49L25}







\maketitle
\begin{abstract}
Given  a nonlinear control system, a target set,  a nonnegative integral cost, and a  continuous function $W$, we say that   the system   is  {\em globally asymptotically controllable to the target  with $W$-regulated cost}, whenever, starting from any point $z$, among the strategies that achieve classical asymptotic controllability  we can select one that also keeps the cost less than $W(z)$.
  In this paper, assuming mild regularity hypotheses on the data, we prove that a necessary and sufficient condition for  global asymptotic controllability with regulated cost is  the existence of a special, continuous Control Lyapunov function, called a {\em Minimum Restraint function}. The main novelty is the necessity implication, obtained here for the first time. 
Nevertheless, the sufficiency condition extends previous results based on semiconcavity of the Minimum Restraint function,   while we require mere continuity.   \end{abstract}


\section{Introduction}
In this paper we extend the classical  equivalence result between global asymptotic controllability to a closed set $\C\subseteq\R^n$  and the existence of a Control Lyapunov  function, to the case where the control system is associated with a cost  which has to be {\em regulated} (i.e., loosely speaking,  kept bounded).  In particular,   we consider a nonlinear control system of the form 
\begin{equation}\label{Eintro}
\dot x(t)=f(x(t),u(t)), \qquad x(0)=z\in\R^n\setminus\C, \qquad u(t)\in U\subseteq\R^m,
\end{equation}
and an  integral cost  
\begin{equation}\label{minprobintro}
\int_ 0^{{T}_{z}(u) } l(x(t),u(t))\, dt,
\end{equation}
where the running cost  $l$ is nonnegative  and the time $T_z(u)\le+\infty$  satisfies\footnote{We will consider assumptions under which, given $z$ and $u$, the corresponding solution to \eqref{Eintro} is uniquely determined.}
\begin{equation}\label{Cintro}
 x(t)  \in \R^n\setminus\C \ \ \text{for all $t\in[0,T_z(u))$}, \quad \lim_{t\to {T}^-_{z}(u)}  \d (x(t))=0
\end{equation}
(for any $y\in\R^n$,  $\d(y)$ denotes the  distance of  $y$ from   $\T$).  As customary, the control system \eqref{Eintro} is called {\em globally asymptotically controllable   to $\C$}   if there exists a $\K\L$ function $\beta$  such that, for any initial point $z$, there is a measurable control $u$   whose corresponding solution to \eqref{Eintro} satisfies  
$\d(x(t))\leq \beta(\d(z),t)$ for all  $t\geq 0$. If, in addition, there exists a  continuous, proper, and positive definite  function  $W:\overline{\R^n\setminus\T}\to[0,+\infty)$,  such that 
$$
\int_0^{{T}_{z}(u)  }l(x(t),u(t))\,dt\le W(z),
$$
we say that \eqref{Eintro}-\eqref{minprobintro}  is globally asymptotically controllable   to $\C$ {\em with $W$-regulated} (or simply, {\em regulated}) {\em cost}.  Slightly extending the original definition   in  \cite{MR13}, we define  as {\it Minimum Restraint function}     any function $V:\overline{\R^n\setminus\C}\to\R$,  which is continuous, positive definite, and proper, and satisfies the {\em decrease condition}
$$
\min_{u\in U}\Big\{\langle p\,,\,f(z,u)\rangle+\p(V(z))\,l(z,u)\Big\}\le-\gamma(V(z)) \ \ \text{for all} \ z\in\R^n\setminus\C, \   p\in \partial_PV(z),
$$
where  $\partial_PV(z)$ is the proximal subdifferential, for some continuous  increasing functions $p_0:(0,+\infty)\to[0,1]$  and  $\gamma:(0,+\infty)\to(0,+\infty)$. Since $p_0$ and $l$ are nonnegative,  this decrease condition implies the classical relation which characterizes Control Lyapunov functions, so that a Minimum Restraint function is actually a special Control Lyapunov function.
\vsm
The main result of the paper is a `Converse  Lyapunov-type theorem', which consists in  proving  that  the existence of  a continuous Minimum Restraint function for some  $p_0$  such that $1/p_0$ is integrable at $0^+$,  is  a necessary and sufficient condition for  global asymptotic controllability with regulated cost.
\vsm
In the case without cost (i.e., for $l\equiv 0$), the equivalence between asymptotic controllability to a point or to a set  of a nonlinear control system  and the existence of a (possibly, non differentiable) Control Lyapunov function,  has been a central topic  in control theory since  the 1980s and nowadays it is established  under very general assumptions (see e.g. the survey papers \cite{C10,K15}, the references therein,  and a recent extension to impulsive control systems  \cite{LM22}).     
The  key idea of many   converse Lyapunov theorems  is to convert the control system into a differential inclusion (see, for instance,   \cite{CLRS,R00}, where $\C=\{0\}$,  and \cite{KT04},  for  a general target) and then use the result on the existence of a Lyapunov function for the differential inclusion to get the promised Control Lyapunov function.  However,  it seems difficult to adapt this approach to the case with a cost,  which we need to estimate from above. Also note  that the cost cannot be treated as an additional state variable, because it is  increasing and cannot be associated with any target. 
\vsm
 In this paper we are inspired instead by  early work \cite{S83}, in which a  suitable cost is associated with the original control system, whose value function turns out to be a continuous Control Lyapunov function. However, even following this approach,  the generalization of the nonsmooth Lyapunov literature to allow for the presence of the cost  \eqref{minprobintro} is far from trivial. In particular, we cannot, as one might think,  simply add to the cost considered in  \cite{S83} the current cost $l$. In fact,  the presence of the cost \eqref{minprobintro} substantially changes  the whole construction of both the  function $\beta$, characterizing asymptotic  controllability,  and the cost used to obtain a Control Lyapunov function.

 The proof technique is also novel, in that it is based on a combination of the classical Lyapunov procedures adopted in \cite{S83} and  a viscosity solutions  approach.  Thanks to this, differently from \cite{S83}, in the proof of necessity we do not need to use relaxed controls and exhibit an explicit construction of the $\K\L$ function $\beta$ and of the  bound $W$ on the cost.  Specifically,  starting from the observation that a continuous, positive definite and proper function $V:\overline{\R^n\setminus\C}\to\R$ solves the decrease condition if and only if it is a viscosity supersolution of the Hamilton-Jacobi-Bellman equation
$$
\max_{u\in U}\Big\{-Dv(z)\,,\,f(z,u)\rangle-[\p(v(z))\,l(z,u)+ \gamma(V(z))]\Big\}=0  \quad \ \text{for all} \  z\in\R^n\setminus\C,
$$
we derive both the facts that the value function built in the proof of the necessity implication satisfies the decrease condition  and the sufficiency implication,  from a viscosity super-optimality principle (Proposition \ref{Th_supsol} below). However, this principle  is not included in the known theory, since $p_0$ is merely continuous and we do not assume the usual  linear growth hypothesis on the dynamics function $f$, but $x$-local Lipschitz continuity only  (see  \cite[Thm. 2.40]{BCD} or  \cite[Thm. 3.3]{GS04}). We  prove it  by introducing a slight generalization of a classical comparison principle for infinite horizon problems (Lemma \ref{LCP} below),  interesting in itself.
\vsm 
We leave for future investigation the issue of the existence of a semiconcave Minimum Restraint function, which, as is well known, plays a key role in the feedback stabilizability of nonlinear systems, both with cost (see \cite{LM18,LM19,LM20}, and \cite{FMR1,FMR2} for a notion of degree-$k$ Minimum Restraint function) and without  (see e.g. \cite{R00,R02,KT04,LM22}, the survey paper \cite{C10}  and references therein, and  \cite{MR18,F21}  for a notion of degree-$k$ Control Lyapunov function).
\vsm
 The paper is organized as follows. In the remaining part of this section we give some notations.  In Section \ref{s2} we introduce precisely   assumptions  and definitions and state our Converse  Lyapunov-type theorem.  Section \ref{s3} is devoted to prove that global asymptotic controllability with $W$-regulated cost for some $W$, implies the existence of a continuous Minimum Restraint function, while the converse implication is obtained in   Section \ref{s4}.

\subsection{Notation}\label{Sprel}
For $a,b\in\R$, we set $a\vee b:=\max\{a,b\}$, $a\land b:=\min\{a,b\}$.  Let $\Omega\subseteq \R^N$ for some integer $N\ge1$ be a nonempty set. For every $r\geq 0$, we set $B_r(\Omega):=\{x\in \R^n\mid \ d(x,\Omega)\leq r\}$, where $d$ is the usual Euclidean distance. We use $\overline\Omega$,  $\partial\Omega$, and  $\mathring \Omega$ to denote the closure, the boundary, and the interior of $\Omega$, respectively.  For any interval $I\subseteq\R$,    $\M(I,\Omega)$,  $AC(I,\Omega)$ are the sets of  functions  $x:I\to\Omega$, which are Lebesgue measurable or absolutely continuous, respectively,  on $I$. 
	 When no confusion may arise, we  simply write $\M(I)$,  $AC(I)$. 
	 \vsm
	 As customary, we   use   ${\mathcal{KL}}$ to denote the set of all continuous functions 
	$\beta:[0,+\infty)\times[0,+\infty)\to[0,+\infty)$ such that:
	(1)\, $\beta(0,t)=0$ and $\beta(\cdot,t)$ is strictly increasing and unbounded 
	for each $t\ge0$;
	(2)\, $\beta(r,\cdot)$ is strictly decreasing  for each $r\ge0$; (3)\, $\beta(r,t)\to0$ 
	as $t\to+\infty$ for each $r\ge0$.
	\vsm
	Given an open set $\Omega\subseteq\R^N$, a continuous function $W:\overline{\Omega} \to[0,+\infty)$ is said {\em positive definite} if  $W(x)>0$ \,$\forall x\in\Omega$ and $W(x)=0$ \,$\forall x\in\partial\Omega$. It is called {\em proper }  if the pre-image $W^{-1}(K)$ of any compact set $K\subset[0,+\infty)$ is compact.   
 Let $x\in\Omega$. The set
$$
D^-W(x):=\left\{p\in\R^n\mid \ \liminf_{y\to x}\frac{W(y)-W(x)-p(y-x)}{|y-x|}\ge0\right\},
$$ 
is  the (possibly empty)  {\em viscosity subdifferential of  $W$ at $x$}.  We recall that  $p\in D^-W(x)$ if and only if there exists $\varphi\in C^1(\Omega)$  such that $D\varphi(x)=p$ and  $W-\varphi$  has a local minimum at $x$ (see e.g. \cite{BCD}). We use $\partial_PW(x)$  to denote the {\em proximal subdifferential of $W$ at $x$} (which may very well be empty).  As it is  known,   $p$ belongs to $\partial_PW(x)$  if and only if there exist $\sigma$ and $\eta>0$ such that
$$
W(y)-W(x)+\sigma|y-x|^2\ge \langle p\,,\, y-x\rangle \qquad \ \text{for all} \  y\in  B_\eta(\{x\}).
$$
 The {\em limiting subdifferential  $\partial_LW(x)$ of $W$ at $x\in \Omega$,}  is defined as
	$$
	\displaystyle \partial_LW(x) := \Big\{\lim_{i\to+\infty} \,  p_i\mid \  p_i\in \partial_PW(x_i), \ \lim_{i\to+\infty} x_i=x\Big\}.
	$$
The  set  $\partial_LW(x)$ is always closed.  If $W$ is locally Lipschitz continuous on $\Omega$,  $\partial_LW(x)$ is compact, nonempty at every point, the set-valued map $x\rightsquigarrow \partial_LW(x)$ is upper semicontinuous, and the  Clarke generalized gradient at $x$  coincides with co\,$\partial_LW(x)$.  As sources for  nonsmooth analysis  we refer e.g. to \cite {CS,CLSW,Vinter}.

\vsm
Let $W:\overline{\R^n\setminus\T}\to[0,+\infty)$ be a continuous, proper, and positive definite function. We can relate the level sets of   $W$ with the ones of the distance function $\d$, by introducing the functions  $d_{W^+}$, $d_{W^-}:(0,+\infty)\to (0,+\infty)$, given by
\begin{align}
d_{W^-}(r):=\sup\left\{\alpha>0\mid \ \ \{\tilde z\mid \ W(\tilde z)\le\alpha\}\subseteq  \{\tilde z\mid \  \d(\tilde z) \le r\}\right\},\label{d-def} \\ 
d_{W^+}(r):=\inf\left\{\alpha>0\mid \ \ \{\tilde z\mid \ W(\tilde z)\le\alpha\}\supseteq  \{\tilde z\mid \  \d( \tilde z)\le r\}\right\}\label{d+def}. 
\end{align}
By \cite[Lemma 3.6]{LM20}, these functions are well-defined, increasing,  and  
\begin{equation}\label{Llim}
\lim_{r\to 0^+}d_{W^+}(r)=\lim_{r\to 0^+}d_{W^-}(r)=0, \quad \lim_{r\to +\infty}d_{W^+}(r)=\lim_{r\to +\infty}d_{W^-}(r)=+\infty.
\end{equation}
Moreover, one has
\begin{equation}\label{Ldis}
d_{W^-}(\d(x))\le W (x)\le d_{W^+}(\d(x)) \quad \ \text{for all} \    x\in\R^n\setminus\C.
\end{equation}
Approximating $d_{W^-}$ from below and  $d_{W^+}$ from above if necessary, we can  thus assume the existence of  continuous, strictly increasing functions, still denoted  $d_{W^-}$ and $d_{W^+}$,   satisfying  \eqref{Llim} and \eqref{Ldis}.


\section{A Converse  Theorem for Minimum Restraint functions}\label{s2} 
Throughout the whole paper we assume that: 
{\em\begin{itemize}
\item[{\rm (i)}] $U\subset\R^m$ is a nonempty compact set,  $\C\subset\R^n$ is a nonempty, closed subset with compact boundary;
\item[{\rm (ii)}]  the functions $f:\R^n\times U\to\R^n$, $l:\R^n\times U\to[0,+\infty)$ are   continuous on $\R^n\times U$,  $x\mapsto f(x,u)$ and $x\mapsto l(x,u)$  are locally Lipschitz continuous, uniformly with respect to $u\in U$.  
\end{itemize}}
\vsm 
Under these assumptions, given   $z\in \R^n\setminus\T$ and   $u\in \M([0,+\infty),U)$, there exist a maximal time $T^{max}\le+\infty$ and  a unique solution $x\in AC([0,T^{\max}),\R^n)$ such that  $x(0)=z$ and 
\begin{equation}\label{Egen}
	  \dot x(t)=f(x(t),u(t)), \qquad\text{a.e. $t\in[0,T^{\max})$.} 
	 \end{equation} 
This solution (or trajectory) will be denoted $x(\cdot\,,u,z)$.  If, in addition, some $c\ge0$ is given, we define the corresponding cost $ x^0(\cdot\,,u,c,z)$ as
 \begin{equation}\label{Cgen}
	 x^0(t,u,c,z):=c+\int_0^tl(x(t,u,z),u(t))\,dt \qquad \ \text{for all} \   t\in[0,T^{\max}).
	 \end{equation}

Let us preliminarily introduce some definitions. 
\begin{definition}[Admissible controls, trajectories, and costs]\label{Admgen} 
 Given an initial condition $z\in \R^n\setminus\T$, a control  $u\in \M([0,+\infty),U)$ is called {\em admissible from $z$} if  there exists $0<T_z(u)\le T^{\max}\le +\infty$ such that 
 $$
 x(t):=x(t,u,z)\notin \C \quad\ \text{for all} \   t\in[0,T_z(u)); \quad     \lim_{t\to T^-_z(u)}\d(x(t))=0 \quad\text{if $T_z(u)<+\infty$.}
 $$
  The set of admissible controls from $z$ will be denoted $\U(z)$. Given $z\in \R^n\setminus\T$, $c\ge0$, and a  control $u\in \U(z)$, we will extend the corresponding trajectory $x=x(\cdot\,,u,z)$ and cost $x^0:=x^0(\cdot\,,u,c,z)$ to $[0+\infty)$, by setting\,\footnote{The limit always exists, as  $\partial\T$ is compact and  $(x^0,x)$ is Lipschitz continuous in any compact set $\overline{B_R(\T)\setminus\T}$, $R>0$. }  
  $$
  (x^0,x)(t):=\lim_{t\to T^-_z(u)}(x^0,x)(t) \qquad\text{for any $t\geq T_z(u)$.}
  $$
 We will call  $(x,u)$ and  $(x^0,x,u)$ (both defined on $[0,+\infty)$)  an {\em admissible pair from $z$} and   an {\em admissible triple from $(c,z)$,}  respectively.
   \end{definition}
We recall the notion of global asymptotic controllability with regulated cost, first introduced in \cite{MR13}.  
In the following,   we will  refer to any function $\beta\in{\mathcal{KL}}$   as a {\em descent rate}. 
	
	\begin{definition}[GAC with regulated cost]\label{DGAC}  The control system \eqref{Egen}  is {\em globally asymptotically controllable  -- in short,  GAC --   to $\C$}   if there exists a descent rate $\beta$ such that, for any initial point $z\in\R^n\setminus \C$, there is an admissible pair $(x,u)$ from $z$,  satisfying
\begin{equation}\label{Dd}
\d(x(t))\leq \beta(\d(z),t) \qquad \ \text{for all} \   t\geq 0.
\end{equation}
If, in addition, there exists a  continuous, proper, and positive definite  function  $W:\overline{\R^n\setminus\T}\to[0,+\infty)$,  such that  there is an  admissible triple $(x^0,x,u)$ from $(0,z)$ associated with a pair  $(x,u)$ for which \eqref{Dd} is valid, also  satisfies
 \begin{equation}\label{DWreg}
x^0(t)=\int_0^{t }l(x(s),u(s))\,ds\le W(x)  \qquad \ \text{for all} \   t\geq 0,
\end{equation}
we say that \eqref{Egen} with the cost \eqref{Cgen}  is {\em globally asymptotically controllable    to $\C$ with $W$-regulated cost}. 
In this case, we will often  simply say that  \eqref{Egen}-\eqref{Cgen}  is GAC (to $\T$) with regulated cost.
\end{definition}

Given a descent rate $\beta$ and a continuous, proper, and positive definite function $W:\overline{\R^n\setminus\T}\to[0,+\infty)$,  for any $z\in\R^n\setminus\T$   such that $\U(z)\ne\emptyset$,   we  set
 $$
\U_{\beta,W}(z):=\left\{
\begin{array}{l} u\in\U(z)\mid   \text{the admissible triple $(x^0,x,u)$  from $(0,z)$  satisfies}  \\[1.5ex]
\ \qquad \qquad  \text{$\d(x(t)) < \beta(\d(z),t)$ and $x^0(t)\le W(z)$ for all $t\ge0$}
   \end{array}
   \right\}.
$$  
{Since in the definition of GAC it is clearly equivalent to replace the ``$\le$" in \eqref{Dd} with  ``$<$", when  \eqref{Egen}-\eqref{Cgen} meet  the properties in Def. \ref{DGAC} for some $\beta$, $W$,  we can assume without loss of generality that $\U_{\beta,W}(z)\ne\emptyset$ for all $z\in\R^n\setminus\C$. }
 
 \vsm
Let us now give a slightly extended version of the notion of  Minimum Restraint  function   for  \eqref{Egen}-\eqref{Cgen}, firstly introduced in  \cite{MR13}. To this end,  we consider the
Hamiltonian  
\begin{equation}\label{Ham}
\ds H(x,\p,p):=\min_{u\in U}\big\{\langle p\,,\,f(x,u)\rangle+\p\,l(x,u)\big\}.
\end{equation}
 
 \begin{definition}[MRF]\label{defMRF}
   Let $W:\ol{\R^n\setminus {\T}}\to[0,+\infty)$ be a continuous function, which is    positive definite  and  proper. We say that $W$  is a \emph{Minimum Restraint function -- in short,  MRF --} for  \eqref{Egen}-\eqref{Cgen} if  it satisfies the   {\emph{ decrease condition}:}\,\footnote{This  means that $H(x,\p(W(x)), p )\le-\gamma(W(x))$ for every $p\in \partial_P W(x)$, where $\partial _P W(x)$ is the proximal differential of $W$ at $x$ (see  Subsection \ref{Sprel}).} 
    \begin{equation}\label{MRH} 
    H (x,\p(W(x)), \partial_P W(x) )\le-\gamma(W(x)) \quad \ \text{for all} \   x\in {{\R^n}\setminus\T},   
    \end{equation}
 for some  continuous, increasing function $\p:(0,+\infty)\to[0,1]$ and some continuous, strictly increasing function $\gamma:(0,+\infty)\to(0,+\infty)$.
\end{definition}
As noted in the introduction, a  MRF is a particular Control Lyapunov function, in which the classical decrease condition 
$$
\min_{u\in U} \langle  \partial_P W(x) \,,\,f(x,u)\rangle\le-\gamma(W(x)) \quad \ \text{for all} \   x\in {{\R^n}\setminus\T},
$$
 is replaced by the stronger  condition \eqref{MRH},  also involving the current cost $l$.

Finally,  we consider the following integrability condition.
\begin{definition}\label{Def_p_0}
Let $p_0:(0,+\infty)\to[0,1]$ be an increasing, continuous  function. We say that $p_0$ satisfies the {\em  integrability condition {\rm (IC)}}, when $1/p_0$ is integrable at $0^+$, namely,  we can define the $C^1$, strictly increasing   function $P:[0,+\infty) \to[0,+\infty)$, given by
\begin{equation}\label{P}
P(v):=\int_0^v\frac{dv}{p_0(v )} \qquad\ \text{for all} \   v\ge0,
\end{equation}
 and, moreover, $P$ satisfies $\ds\lim_{v\to+\infty}P(v)=+\infty$.
\end{definition}
Clearly, if $p_0$ is a positive constant, as in the original definition of MRF, it trivially satisfies  condition (IC).
We are now ready to state our main result.
\begin{theorem}[Converse MRF Thm.]\label{thm3}
The following properties are equivalent:
\begin{itemize}
\item[{\rm (i)}] system \eqref{Egen} with cost \eqref{Cgen} is GAC to $\C$ with regulated cost;
\item[{\rm (ii)}] there exists a continuous    MRF for \eqref{Egen}-\eqref{Cgen}, for some $p_0$ and $\gamma$  such that $p_0$ satisfies the integrability condition {\rm (IC)}.
\end{itemize}
 \end{theorem}
 The rest of the paper is devoted to the proof of Theorem \ref{thm3}.

 \begin{remark}\label{Rpde}   Starting with the assumption that the system is GAC with regulated cost, in the proof below we will explicitly build a  MRF with $p_0\equiv1$,  as   unique continuous viscosity solution of the Hamilton-Jacobi-Bellman equation (in short, HJB) associated with an exit-time problem with vanishing lagrangian.  As is well known, these HJB equations are highly degenerate and have in general multiple solutions, for which the continuity on the target does not propagate to the whole domain (see \cite{MS15} and  \cite{Sor99,Ma04,M04}). The  proof technique  thus consists  in showing the continuity of the solution and establishing an ad hoc comparison principle.  In addition to allowing  us to extend the main result of \cite{S83}  to the case with cost, this technique also provides an alternative approach to obtaining the classical result.   
 \end{remark}

\begin{remark}\label{RTmin}  In this paper we continue the study, begun with \cite{MR13}, aimed at constructing a unified theory, which has as extreme situations asymptotic controllability (with cost $l\equiv0$) on the one hand,  and the minimum time problem (with cost $l\equiv1$) on the other, and for which the $l\ge0$ case represents, in a sense, the intermediate stage.   In the original notion of MRF in \cite{MR13}, $\p$ was a positive constant.  Extending the definition by considering  $p_0$  an increasing function, possibly vanishing at the origin but satisfying the integrability condition  {\rm (IC)},   generalizes the cost bound obtained in \cite{MR13}, as it implies GAC with $\bar W$-regulated cost, where $\bar W(x)=4\,P(W(x)/2)$  for $P$ as in \eqref{P} (see estimate \eqref{P_est} below).\footnote{Actually, refining some estimates, we could likely get $\bar W(x)=P(W(x))$, as a consequence of the fact  that  $\bar W(x)\le (1+2\varepsilon) (P(W(x)/(1+\varepsilon)))$ for every $\varepsilon>0$.} In the special case of the  minimum time problem,  this extension finally provides   a result which is entirely consistent with the existing  literature. Indeed, let $l\equiv1$  and assume that  
 the distance function $\d$ is a  MRF  for some functions $\p$ and  $\gamma$  such that {\rm (IC)} holds true. Then, from the decrease condition \eqref{MRH},   it follows that  $\d$  satisfies 
\begin{equation}\label{Tdiss}
 \min_{u\in U} \langle  \partial_P \d(x) \,,\,f(x,u)\rangle \leq -\tilde\gamma(\d(z))\qquad  \ \text{for all} \   z\in\R^n\setminus\T,  
\end{equation}
where $\tilde\gamma(r):=p_0(r)+\gamma(r)$. As it is well-known, this condition combined with the integrability  property {\rm (IC)} for $\tilde\gamma$ --sometimes called   {\em weak Petrov   condition}--    guarantees the small time local controllability of system \eqref{Egen}  to   $\T$, and implies for the minimum time function $T$  the estimate $T(z)\le\int_0^{\d(z)}(1/\tilde\gamma(r))\,dr$,  in line with our result \eqref{P_est} (see e.g. \cite{CS}). By the expression ``weak", we mean that $\tilde\gamma$ can be 0 at 0, to distinguish it from the classical Petrov condition, in which $\tilde\gamma$ is replaced by a positive constant.   Notice  that, considering only $p_0\equiv \bar p_0>0$  constant, we would have $\tilde \gamma\ge \bar p_0>0$, so our conditions would include just  the ordinary, i.e. non-weak,  Petrov condition, which implies  local Lipschitz continuity of $T$ (see e.g. \cite{BCD}).

Furthermore, we think that considering $\p$ not constant could play a  role in regularizing a MRF,   in the fashion of \cite{R00}.
\end{remark}

 \section{Proof of implication {\rm (i)} $\Longrightarrow$ {\rm (ii)}}\label{s3}
Suppose that  \eqref{Egen}-\eqref{Cgen} is  GAC to $\C$ with $W$-regulated cost, for a cost bound $W$ and a descent rate $\beta$. 
 We   split the proof into several lemmas. In particular, in   Lemma \ref{l1}, we build new functions $\bar\beta$ and $\bar W$ and an admissible triple $(\hat x^0,\hat x,\hat u)$ from $(0,z)$ such that  $\hat u\in\U_{\bar\beta,\bar W}(z)$, for every $z\in\R^n\setminus\C$.  These objects  play  a key role in the construction of a (larger) cost  functional $J$,   in Lemma  \ref{l2}. Furthermore, in Lemmas \ref{l3}, \ref{l4} we show that the  value function $V$ associated with the  control system  \eqref{Egen} and the new cost $J$,  is a continuous MRF.  Lemmas \ref{l1}, \ref{l2} are inspired by \cite[Lemmas 3.8,\,3.17]{S83}, where, however,  no cost is considered and relaxed controls are used. As already observed, differently from \cite{S83},    the present results are formulated in terms of  (explicitly built) new descent rate and cost bound, and  are obtained  by mixing nonsmooth analysis and viscosity  methods, which incidentally allow us to disregard relaxed controls. 
 
 \vsm 
  We begin by introducing some definitions. We consider a bilateral sequence $(r_i)_{i\in\Z}$, given by\,\footnote{ As $\beta^{-1}(r,0)$, we mean the inverse of the strictly increasing function $r'\mapsto r= \beta(r',0)$.}
 \begin{equation}\label{ri1}
r_0:=1, \qquad    
r_{i }:=\min\left\{\beta^{-1}(r_{i-1},0), d^{-1}_{W^+}\left(\frac{1}{4}d_{W^-}(r_{i-1})\right)\right\} \quad \ \text{for all} \   i\in\Z.
\end{equation}
 Clearly,  $(r_i)_{i\in\Z}$  is  positive, strictly decreasing, so that $r_1<1$, and satisfies 
$$\displaystyle\lim_{i\to-\infty}r_i=+\infty, \qquad \displaystyle\lim_{i\to+\infty}r_i=0.$$ 
Hence,  we have
$$\beta(r_{i},0)\le r_{i-1},\quad d_{W^+}(r_i)\leq \frac{1}{4}d_{W^-}(r_{i-1}) \quad \ \text{for all} \   i\in\Z$$
and, consequently, recalling that   $d_{W^-}\le d_{W^+}$  (see \eqref{d-def} and \eqref{d+def}), 
\begin{equation}\label{giterate}
d_{W^+}(r_{i+N})\leq \frac{1}{4^N}d_{W^-}(r_{i})\leq \frac{1}{4^N}d_{W^+}(r_{i}) \quad \ \text{for all} \   i\in\Z,\ \ \text{for all} \   N\in\N, \  N\ge1.
\end{equation}
 For any $i\in\Z$, let 
 \begin{equation}\label{B_i}
 \mathcal B_i:= \{z\in\R^n\setminus\C\mid \  \d(z)\in  [r_i,r_{i-1}]\}, 
\end{equation}
   so that $\displaystyle\R^n\setminus\C=\cup_{i\in\Z} \mathcal B_i$. Finally,  we define the \emph{$i$-th $(\beta,W)$-strip} $\mathcal A_i$, as 
 $$
 \mathcal A_i:=\left\{\begin{array}{l} (x^0,x,u,z)\mid (x^0,x,u) \ \text{admissible triple from $(0,z)$,}  \\
   \qquad \qquad \qquad \qquad \qquad \qquad\qquad u\in\U_{\beta,W}(z),  \ z\in \mathcal B_i
   \end{array}\right\}.
 $$
 
 \begin{lemma} \label{l1}   
There exist a $\K\L$ function $\bar\beta \ge\beta$,  a continuous, unbounded, strictly  increasing map $\Phi:[0,+\infty)\to [0,+\infty)$ with $\Phi(0)=0$, and a function $T:(0,+\infty)\to[0,+\infty)$ with $T(R)=0$ for all $R\le r_1$,\footnote{The value $r_1$ is  defined as in \eqref{ri1}.}
 such that 
for any $z\in \R^n\setminus\T$ there exists an admissible triple $(\hat x^0, \hat x, \hat u)$ from $(0,z)$ enjoying the following properties:
\begin{enumerate}[{\rm (i)}]
\item $\d(\hat x(t))\leq \bar\beta(\d(z),t)$  for all $t\ge0$;
\item $\hat x^0(t)\leq \bar W(z):=\Phi(\d(z))$ for all $t\geq0$;
\item $\d(\hat x(t))\leq \bar\beta(1, t-T(\d(z)))$ for all $t\geq T(\d(z))$.

\end{enumerate}
 \end{lemma}

\begin{proof} 
 {\em Step 1  (Properties of the $(\beta,W)$-strips).}
Fix $i\in\Z$ and let $(u^0,x,u,z)\in\mathcal A_i$.   From the definitions of  $(r_i)_{i\in\Z}$,  $d_{W^-}$, and $d_{W^+}$, it follows that  
\begin{equation}\label{eq1l1}
\d(x(t))< \beta(\d(z),t)\leq \beta (r_{i-1},0)< r_{i-2}\qquad \text{for all }t\geq 0,
\end{equation}
and
\begin{equation}\label{eq2l1}
W(z)\leq d_{W^+}(r_{i-1})\leq \frac{1}{4}d_{W^-}(r_{i-2}).
\end{equation}
Define
$$T_{i,z}:=\inf \left\{t\geq0 \mid \d(x(t))=\frac{r_{i}+r_{i+1}}{2}\right\}.$$
Clearly,  $0<T_{i,z}<T_z(u)$, {where $T_z(u)$ is as in Def. \ref{Admgen}.}
Set
$$\tilde \varepsilon_{i,z}:=\inf\left\{\frac{1}{2}\big(\beta(\d(z),t)-\d(x(t))\big)\mid t\in[0,T_{i,z}]\right\}$$
Note that by the continuity of $\beta$ and of $x$, $\tilde \varepsilon_{i,z}$ is actually a minimum and $\tilde \varepsilon_{i,z}>0$. Also define $\hat \varepsilon_{i}:=\frac{1}{4}d_{W^+}(r_{i-1})$, $ \bar\varepsilon_i:=\frac{r_i-r_{i+1}}{4}$
and, finally, set
$$\varepsilon_{i,z}:=\min\{\tilde \varepsilon_{i,z},\hat \varepsilon_i,\bar \varepsilon_i\}.$$
Since by assumption $f(\cdot,u)$ and $l(\cdot,u)$ are locally Lipschitz continuous, uniformly with respect to $u\in U$, then
 there exists $\delta_{i,z}>0$ such that, for all $\bar z\in \R^n\setminus\T$ satisfying $|\bar z-z|<\delta_{i,z}$,  the cost $x^0(\cdot\,,u,0, \bar z)$ and the trajectory  $x(\cdot\,,u,\bar z)$ are defined on $[0,T_{i,z}]$ and  satisfy
$$|x(t)-x(t,u,\bar z)|\leq \varepsilon_{i,z},\quad |x^0(t)-x^0(t,u,0, \bar z)|\leq \varepsilon_{i,z}\qquad \text{for all } t\in [0,T_{i,z}].$$
Hence,  for all $t\in[0,T_{i,z}]$,   using \eqref{eq2l1},   we get 
\begin{align*}  
x^0(t,u,0,\bar z)&\leq x^0(t)+\varepsilon_{i,z}\leq W(z)+\hat \varepsilon_i \leq W(z)+\frac{1}{4}d_{W^+}(r_{i-1})\\
&\leq \frac{1}{4}d_{W^-}(r_{i-2})+\frac{1}{4}d_{W^+}(r_{i-1}) \leq \frac{1}{2}d_{W^+}(r_{i-2}).
\end{align*}
Moreover, in view of the definition of $\varepsilon_{i,z}$,  we also have  
\begin{align*}
    \d(x(t,u,\bar z))&\leq \d(x(t))+\frac{1}{2}\big(\beta(\d(z),t)-\d(x(t))\big)=\frac{1}{2}\beta(\d(z),t)+\frac{1}{2}\d(x(t))\\
    &<\beta(\d(z),t)\leq \beta(r_{i-1},0)\leq r_{i-2},
\end{align*}
whereas the definition of $T_{i,z}$ implies
\begin{align*}
    \d(x(t,u,\bar z))&\geq \d(x(t))-\bar \varepsilon_{i}\geq \d(x(T_{i,z}))-\bar \varepsilon_i\\
    &\geq \frac{r_i+r_{i+1}}{2}-\frac{r_i-r_{i+1}}{4}=\frac{r_i}{4}+\frac{3}{4}r_{i+1}>r_{i+1},
\end{align*}
and
 \begin{align*}
    \d(x(T_{i,z},u,\bar z))&\leq \d(x(T_{i,z}))+\bar \varepsilon_{i}= \frac{r_i+r_{i+1}}{2}+\frac{r_i-r_{i+1}}{4} \\
    &=\frac{3}{4}r_i+\frac{1}{4}r_{i+1}<r_{i}.
\end{align*}
Summarizing the above results, we can conclude that, for every $i\in\Z$ and $z\in \mathcal B_i$ ($\mathcal B_i$ as in \eqref{B_i}),  with which we can consider an element  $(x^0,x,u,z)\in \mathcal A_i$ to be associated,
 by the axiom of choice,  there exists $\delta_{i,z}>0$ such that, for all $\bar z\in \R^n\setminus \C$ with $|z-\bar z|<\delta_{i,z}$, one has
\begin{enumerate}[(a)]
    \item $\d(x(t,u,\bar z))\in (r_{i+1},r_{i-2})$ for all $t\in[0,T_{i,z}]$;
    \item $\d(x(T_{i,z},u,\bar z))\in (r_{i+1},r_i)$, i.e. $x(T_{i,z},u,\bar z)\in \mathring{\mathcal B}_{i+1}$;
    \item $x^0(T_{i,z},u,\bar z)\leq \frac{1}{2}d_{W^+}(r_{i-2})$.
\end{enumerate}
\noindent{\em Step 2 (Construction of a suitable admissible triple)}  
Preliminarily, observe that, since $\partial \T$ is compact, for every  $i\in\Z$ the set  $\mathcal B_i$
 is compact. Therefore, the cover of $\mathcal B_i$ given by the open balls $\mathring B_{\delta_{i,z}}(\{z\})$, $z\in \mathcal B_i$, admits a finite subcover corresponding to the points $z\in Z_i$,  for some finite subset $Z_i$ of $\mathcal B_i$. Fix a positive bilateral  sequence $(T_i)_{i\in\Z}$ such that
\begin{equation}\label{T_i}
 T_i\ge\max\{T_{i,z}\mid z\in Z_i\}, \quad  \qquad \sum_{j=0}^\infty T_{i+j}=+\infty, \quad \text{for every $i\in\Z$.} 
\end{equation}
 Furthermore, thanks to the properties of the bilateral sequence $(r_i)_{i\in\Z}$, we can define  the map $i:(0,+\infty)\to\Z$, given by
\begin{equation}\label{def_i}
i(r)=i \qquad \text{if $r\in(r_i,r_{i-1}]$.}
\end{equation}
Fix now $\bar z\in \R^n\setminus \T$ and let $i:=i(\d(\bar z))$. Then, $\bar z\in \mathring B_{\delta_{i,z_{_0}}}(\{z_{_0}\})$ for some $z_{_0}\in Z_i\subset \mathcal B_i$. Let $(x^0_{_0},x_{_0},u_{_0},z_{_0})\in \mathcal A_i$ be the associated process from $(0,z_{_0})$. Since 
 $|\bar z-z_{_0}|<\delta_{i,z_{_0}}$,   from Step 1 it follows that  $\hat x^0:=(\cdot\,,u_{_0},0,\bar z)$ and $\hat x:=x(\cdot\,,u_{_0},\bar z)$ are defined on the interval $[0,\hat t_{_0}]$,  $\hat t_{_0}:=T_{i,z_{_0}}\leq T_i$, and satisfy
\begin{enumerate}[(a.0)]
    \item $\d(\hat x(t))\in (r_{i+1},r_{i-2})$\  for all $t\in[0,\hat t_{_0}]$;
    \item $\hat x(\hat t_{_0})\in \mathring{\mathcal B}_{i+1}$;
    \item $\hat x^0(\hat t_{_0})\leq \frac{1}{2}d_{W^+}(r_{i-2})$.
\end{enumerate}
 Repeating the above procedure with the initial conditions $\bar z_1:=\hat x(\hat t_0)\in \mathcal B_{i+1}$ and  $\bar c_1:=\hat x_0(\hat t_0)$,    we get the existence of an admissible control $u_1\in \U_{\beta,W}(\bar z_1)$ and of a time $\hat t_1\leq T_{i+1}$,  such that
  extending
  $\hat x^0$ and  $\hat x$  by setting $\hat x^0(t)=x^0(t-\hat t_{_0},u_1, \bar c_1, \bar z_1)$ and $\hat x(t)=x(t-\hat t_{_0}, u_1,\bar z_1)$ for $t\in[\hat t_{_0},\hat t_{_0}+\hat t_1]$, respectively,  one has
\begin{enumerate}
    \item[(a.1)] $\d(\hat x(t))\in (r_{i+2},r_{i-1})$ for all $t\in[\hat t_{_0},\hat t_{_0}+\hat t_1]$;
    \item[(b.1)] $\hat x(\hat t_{_0}+\hat t_1)\in \mathring{\mathcal B}_{i+2}$;
    \item[(c.1)] $\hat x^0(\hat t_{_0}+\hat t_1)\leq \frac{1}{2}d_{W^+}(r_{i-2})+\frac{1}{2}d_{W^+}(r_{i-1})\leq \frac{1}{2}d_{W^+}(r_{i-2})+\frac{1}{8}d_{W^+}(r_{i-2}).$
\end{enumerate}
Of course, setting $\hat u(t):=u_{_0}(t)\chi_{_{[0,\hat  t_{_0}]}}(t)+ u_1(t-\hat t_{_0})\chi_{_{(\hat  t_{_0},\hat  t_{_0}+\hat  t_1 ]}}(t)$, we have  $\hat x^0=x(\cdot\,,\hat u, 0,\bar z)$ and  $\hat x=x(\cdot\,,\hat u,\bar z)$. In this way,  we can recursively construct  sequences of 
controls  $(u_N)_{N\in \N}$ and  times $(\hat t_{N})_{N\in\N}$ with $\hat t_{N}\leq T_{i+N}$ for all $N$, such that,  setting
$$\hat T_{-1}:=0,\quad \hat T_{N}:=\sum_{j=0}^N \hat t_{j},\quad \hat T_{\infty}:=\sum_{j=0}^{+\infty} \hat t_{j}$$
and
$$\hat u(t):=u_N(t-\hat T_{N-1})\qquad \text{for all }t\in(\hat T_{N-1},\hat T_{N}]$$
we have $u\in \mathcal M([0,\hat T_{\infty}),U)$ and the corresponding solution $(\hat x^0,\hat x)$ from $(0,\bar z)$ is defined on  the whole interval $[0,\hat T_{\infty})$ and satisfies
\begin{enumerate}
    \item[(a.N)] $\d(\hat x(t))\in (r_{i+N+1},r_{i+N-2})$ for all $t\in[\hat T_{N-1},\hat T_{N}]$;
    \item[(b.N)] $\hat x(\hat T_{N})\in \mathring{\mathcal B}_{i+N+1}$;
    \item[(c.N)] $\hat x^0(\hat T_{N})\leq \hat x_0(\hat T_{N-1})+\frac{1}{2}d_{W^+}(r_{i+N-2})\leq C d_{W^+}(r_{i-2}),$
    for $C:=\frac{1}{2}\sum_{j=0}^\infty \frac{1}{4^j}$. 
\end{enumerate}
  (The latter inequality can be easily proved by induction.) From these relations, it  follows immediately  that $\hat T_\infty=T_z(\hat u)$,  $\lim_{t\to \hat T_\infty^-}\d(\hat x(t))=0$, and $\hat x^0(\hat T_\infty)\le C d_{W^+}(r_{i-2})$ (actually, even if $\hat T_\infty=+\infty$). Hence, if in case  $\hat T_\infty<+\infty$ we extend the process $(\hat x^0,\hat x,\hat u)$ to $[0,+\infty)$ by setting $\hat u(t)=w$ ($w\in U$ arbitrary), and $(\hat x^0,\hat x)(t)=\lim_{t\to \hat T_\infty^-}(\hat x^0,\hat x)(t)$ for any $t\ge \hat T_\infty$, we can conclude that  $(\hat x^0,\hat x,\hat u)$ is an admissible triple from $(0,\bar z)$.

\vsm  
\noindent {\em Step 3  (Construction of    $\bar \beta$, $\Phi$,  and  $T$).} 
 For all $i\in \Z$ and   $N\in \N$,   set
$$\bar T_{i,-1}:=0,\qquad \bar T_{i,N}:=\sum_{j=0}^N T_{i+j},$$
where the times $T_{i}$ are as in \eqref{T_i}, so that
 for any $i\in\Z$, $\bar T_{i,N}\to+\infty$ as $N\to +\infty$. Then, we  define the function $T:(0,+\infty)\to[0+\infty)$,  as
$$
T(R):=\bar T_{i(R),-1\vee (-i(R)+1)} \qquad \ \text{for all} \   R>0,
$$
where  $i(\cdot)$ is as in \eqref{def_i}. Note  that if $R\le r_1<1$ then $i(R)\geq 2$ and this implies $T(R)=\bar T_{i(R),-1}=0$. 
Since  $d_W^+$ as well as  $R\mapsto r_{i(R)}$ are increasing functions,  the composition   $R\mapsto \,d_W^+(r_{i(R)-2})$  is a positive,  piecewise constant, increasing function,  such that $d_W^+(r_{i(R)-2})\to0$ as  $R\to 0^+$ and $d_W^+(r_{i(R)-2})\to+\infty$ as $R\to +\infty$. There is thus a continuous, strictly increasing approximation  from above  $\Phi:[0,+\infty)\to[0,+\infty)$ of this composition times $C$ ($C$ as in (c.N) above), vanishing at  zero and unbounded, namely
 $$
C\, d_W^+(r_{i(R)-2})\le \Phi(R) \qquad\ \text{for all} \   R>0.
 $$
Finally, we introduce the  function $b:[0,+\infty)\times[0,+\infty)\to[0,+\infty)$,  given by
$$\begin{cases}
b(R,t):=r_{i+N-2}& \text{if } (R,t)\in {(r_i,r_{i-1}]} \times [\bar T_{i,N-1},\bar T_{i,N})\quad \ \text{for all} \   i\in\Z, \ N\in\N\\
b(0,t):=0& \ \text{for all} \   t\geq 0.
\end{cases}
$$
Note that $b(\cdot,t)$ is increasing and $b(R,t)\to +\infty$ as $R\to +\infty$, for all $t\geq 0$. Similarly, $b(R,\cdot)$ is positive, decreasing and $b(R,t)\to 0$ as $t\to +\infty$ for all $R> 0$. Using e.g.  a linear interpolation procedure, it is not difficult to show that the discontinuous function  $b$ can be approximated from above by some $\mathcal{KL}$ function $\bar \beta$, which is  $\ge\beta$ by construction. 

\noindent Let $\bar z\in \R^n\setminus \C$ and set $i:=i(\d(\bar z))$. 
Then, since $\hat T_N\le\bar T_{i,N}$ for all $N\ge-1$,  the admissible triple $(\hat x^0,\hat x,\hat u)$  from $(0,\bar z)$ built in Step 2,  satisfies
\begin{equation}\label{hat_x}
\d(\hat x(t))<r_{i+N-2}\cdot\chi_{_{[\bar T_{i,N-1},\bar T_{i,N})}}(t)\leq b(\d(\bar z),t)\le\bar \beta(\d(\bar z),t) \qquad \ \text{for all} \   t\geq 0, \end{equation}
and
$$ \hat x^0(t)\le \bar W(\bar z):=\Phi(\d(\bar z)) \qquad\ \text{for all} \   t\ge0,
$$
so that $\hat u\in\U_{\bar\beta,\bar W}(\bar z)$.  Notice that $\bar W:\overline{\R^n\setminus\T}\to[0,+\infty)$ is a continuous, proper and positive definite function, in view of the properties of $\Phi$ and $\d$. This concludes the proof of  statements (i) and (ii). 

\noindent In order to prove (iii), let us first suppose   $\d(\bar z)\le r_1<r_0=1$. In this case, $i(\d(\bar z))\ge2$ and $T(\d(\bar z))=\bar T_{i(\d(z)),-1}=0$. Hence (iii) follows from \eqref{hat_x}, because $\bar\beta(\cdot, t)$ is strictly increasing for every $t\ge0$, so that   
$$
\d(\hat x(t))< \bar\beta( \d(\bar z),t)\le \bar\beta(1,t)=\bar\beta(1,t-T(\d(\bar z))) \quad\text{ for all $t\ge T(\d(\bar z)=0$.}
$$
 If instead $\d(\bar z)> r_1$,  we have $i:=i(\d(\bar z))\le 1$.  Set
  $N:=-i(\d(z))+2$ \, ($\ge1$) and   $\bar z_N:=\hat x(\hat T_{N-1})$, so that,  by property (b.N) of Step 2,  $r_2<\d(\bar z_N)<r_1<r_0=1$.   Note that,  applying  the above construction   from the initial condition $(0,\bar z_N)$, the obtained admissible triple, say $(\hat x^0 _N, \hat x_N,\hat u_N)$ from $(0,\bar z_N)$, satisfying
  $$
  \d(\hat x_N(s))<\bar\beta(\d(\bar z_N),s), \quad \hat x^0_N(s)\le\Phi(\d(\bar z_N))   \quad\ \text{for all} \   s\ge0,
  $$
  is simply given by 
  $$
  \hat u_{N}(\cdot)=\hat u(\cdot\,+\hat T_{N-1}) , \quad \hat x^0_{ N}=x^0(\cdot\,, \hat u_{ N}, 0,\bar z_{N}), \quad \hat x_{ N}=x(\cdot\,, \hat u_{N}, \bar z_{ N}).
  $$ 
  This fact is crucial for property (iii) to hold. Indeed, using the monotonicities  of $\bar\beta$ and the inequality $\hat T_{N-1}\le\bar T_{i,N-1}=T(\d(\bar z))$,       it implies that
  \begin{align*}
\d(\hat x(t))&=\d(\hat x_{N}(t- \hat T_{N-1}))<\beta(\d(\bar z_N),t- \hat T_{N-1})\\
&<\bar \beta(1,t-\hat T_{N-1})\le \bar \beta(1,t- T(\d(\bar z))) \qquad\text{ for all $t\ge T(\d(\bar z)).$}
\end{align*}
\end{proof}
 
\begin{lemma}\label{l2}
There exist two continuous,  strictly increasing functions  $\ell$, $\Psi:[0,+\infty)\to [0,+\infty)$, with $\ell(0)=0$, $\displaystyle \lim_{R\to+\infty}\ell(R)= +\infty$, such that the functional $J: (\R^n\setminus \C)\times \M([0,+\infty),U)\to[0,+\infty)\cup\{+\infty\}$,  defined as 
$$
J(z,u):=\left\{\begin{array}{l}\displaystyle \int_0^{T_z(u)} \left[\ell(\d(x(t,z,u)))+l(x(t,z,u),u(t))\right]dt \qquad\qquad \text{if $u\in\U(z)$,} \\
+\infty \qquad\qquad\qquad\qquad\qquad\qquad\qquad \text{if $u\in \M([0,+\infty),U)\setminus\U(z)$,}
\end{array}\right.
$$
enjoys  the properties {\rm (i)--(v)} below.

\noindent  For every $z\in \R^n\setminus \C$,  let $(\hat x^0, \hat x,\hat u)$ be the  admissible  triple from $(0,z)$ built in Lemma \ref{l1}. Using  the notations of Lemma \ref{l1}, we have
\begin{enumerate}[{\rm (i)}]
    \item $J(z,\hat u)<+\infty$;
    \item if $\d(z)<r_1$, then $J(z,\hat u)\leq \bar\beta(\d(z),0)+\Phi(\d(z))$;
    \item for any $u\in \mathcal M([0,+\infty),U)$ such that  $J(z, u)\leq J(z,\hat u)$, ($u\in\U(z)$ and)  $\d(x(t, u,z))\leq \Psi(\d(z))$ for all $t\geq 0$;
    \item for all $\alpha>0$ there exists $\Theta>0$ such that  for any   $u\in \mathcal M([0,+\infty),U)$ with $J(z,u)<\alpha$, then  $\d(z)\leq\Theta$;
    \item for all $\alpha>0$ there exists $\delta>0$ such that, if $\d(z)>\alpha$,  then $J(z,u)>\delta$ for all $u\in \mathcal M([0,+\infty),U)$.
\end{enumerate}
\end{lemma}

\begin{proof}
Let $\bar\beta$ be a  $\mathcal {KL}$ function as in Lemma \ref{l1}.  Extend the continuous strictly decreasing function $[0,+\infty)\ni t\mapsto \bar\beta(1,t)$ to a  continuous strictly decreasing function defined on $\R$ and tending to $+\infty$ as $t\to-\infty$. Let $\tau:(0,+\infty)\to\R$ be the inverse of this map, so that,  in particular, $\tau$ is continuous, strictly decreasing,  $\tau(R)\to +\infty$ as $R\to 0^+$ and $\tau(R)\to -\infty$ as $R\to +\infty$. 
    
\medskip
\emph{Step 1. (Construction of a function $\ell_1$)}
Define the  continuous, strictly increasing, and unbounded  function $\ell_1:[0,+\infty)\to[0,+\infty)$, given by
$$
\ell_1(R):=Re^{-\tau(R)} \quad\text{for all $R>0$,} \qquad      \ell_1(0):=0.
$$
For any    $0<b<c$,  we claim that there exists a function {$\varkappa(b,c)$} such that, if $z_1,z_2\in \R^n\setminus \C$, $u\in \mathcal M([0,+\infty),U)$, and  $T>0$ satisfy $z_2=x(T,u,z_1)$ and $b\leq \d(x(t,u,z_1))\leq c$ for all $t\in[0,T]$,  then
\begin{equation}\label{gamma}
\int_0^T \ell_1(\d(x(t,u,z_1)))dt\geq \varkappa(b,c)|z_1-z_2| \ge {\varkappa(b,c)|\d(z_1)-\d(z_2)|.}
\end{equation}
Indeeed, if $z_1=z_2$, \eqref{gamma} is trivially satisfied. If instead $z_1\ne z_2$,  then $\bar M(b,c):=\max\{ |f(x,u)|\mid b\leq \d(x)\leq c,\, u\in U\}>0$ and $|z_1-z_2|{\leq } T \bar M(b,c)$. Setting $\varkappa(b,c):=\ell_1(b)\bar M(b,c)^{-1}$, since $\ell_1$ is increasing, we have $\ell_1(\d(x(t,u,z_1)))\geq \ell_1(b)$ for all $t\in[0,T]$, and  consequently
$$\int_0^T \ell_1(\d(x(t,u,z_1)))dt\geq T\ell_1(b)\geq \varkappa(b,c){|z_1-z_2|}.$$

\medskip
\emph{Step 2 (Recursive definition of a   sequence  $\ell_j$)}
  Starting from $\ell_1$ introduced in Step 1, for every $j\geq 1$ we will  recursively define an increasing sequence of functions $\ell_j:[0,+\infty)\to\R$ with the following property:
\begin{equation}\label{Nj}
0\le R\leq \bar\beta(j,0) \quad \Longrightarrow \quad \ell_{j+1}(R)=\ell_j(R).
\end{equation}
Then, we will prove that the pointwise limit $\displaystyle\ell:=\lim_{j\to+\infty} \ell_j$ (finite,  because of \eqref{Nj}) together with the map $\Psi$ defined below,  satisfies {\rm (i)--(v)}. Further, all the functions $\ell_j$, (hence $\ell$ itself) will be continuous.  We begin by assuming that, for some fixed $j\ge 2$, we already defined the functions $\ell_i$ for all $i\le j$,  satisfying \eqref{Nj} for all $i=1,\dots, j-1$. Given $z\in \R^n\setminus \C$ and any control $u\in\U(z)$, let us set
$$J_j(z,u):=\int_0^{T_z(u)}\ell_j(\d(x(t,u,z)))+l(x(t,u,z),u(t))dt.$$
Observe that, if we consider the admissible triple $(x^0,x,u)$ from $(0,z)$ defined on $[0,+\infty)$ (even in case $T_z(u)<+\infty$, as in Def. \ref{Admgen}),  we can equivalently write
$$
J_j(z,u)= \int_0^{+\infty}\ell_j(\d(x(t)))+\lim_{t\to T^-_z(u)} x^0(t),
$$
as $\d( x(t))=0$ for all $t\ge T_z(u)$, whenever $T_z(u)<+\infty$.
Assume $\d(z)\leq j$.  Let $T(\cdot)$ and  $(\hat x^0, \hat x,\hat u)$ be the uniform time and the admissible triple from    $(0,z)$ built in Lemma \ref{l1}, respectively.  Then,  we have 
$$\d(\hat x(t))\leq \bar\beta(1, t-T(j))\leq  \bar \beta(1,0) \qquad \text{for all $t\geq T(j)$.}
$$
In particular,   \eqref{Nj} (together with the fact that $\bar\beta(\cdot,0)$ is increasing) implies
$$\ell_j(\d(\hat x(t)))=\ell_1(\d(\hat x(t)))\qquad \ \text{for all} \   t\geq T(j).$$
If otherwise $t\leq T(j)$, we get
$$\d(\hat x(t))<\bar\beta(\d(z),0)\leq \bar\beta(j,0).$$
Hence, we have
\begin{align*}
    \int_0^{+\infty}\ell_j&(\d(\hat x(t)))dt = \int_0^{T(j) }\ell_j(\d((\hat x(t))))+ \int_{T(j)}^{+\infty}\ell_j(\d(\hat x(t)))dt\\
    &\leq \ell_j(\bar\beta(j,0))T(j)+ \int_{T(j)}^{+\infty}\ell_1(\d(\hat x(t)))dt\\
    &=   \ell_j(\bar\beta(j,0))T(j)+ \int_{T(j)}^{+\infty}\d(\hat x(t))e^{-t(\d(\hat x(t)))}dt\\
    &\le \ell_j(\bar\beta(j,0))T(j)+ \int_{T(j)}^{+\infty}\bar\beta(1,(t-T(j)))\,e^{-\tau(\bar\beta(1,(t-T(j)))}dt\\
    &\leq  \ell_j(\bar\beta(j,0))T(j)+ \bar\beta(1,0)\int_{T(j)}^{+\infty}e^{-t+T(j)}dt\\
    &=  \ell_j(\bar\beta(j,0))T(j)+ \bar\beta(1,0):= L_j.   
\end{align*}
Therefore, from Lemma \ref{l1},\,(ii)  it follows that 
\begin{equation}
    \label{Rjest}
    J_j(z,\hat u)\leq L_j+\Phi(j) \quad \text{for all } z\in\R^n\setminus\C \text{ such that } \d(z)\leq j.
\end{equation}
Now,  set $\varkappa_j:=\varkappa( \bar\beta(j,0)+1, \bar\beta(j,0)+2)$
    and consider a continuous function $\rho_j:[0,+\infty)\to[0,+\infty)$ such that
    $$\rho_j(R)=\begin{cases}
        0 &\text{ if $R\le \bar\beta(j,0)$ or $R\ge \bar\beta(j,0)+3$},\\
        \dfrac{ L_j+\Phi(j)}{\varkappa_j}&  \text{ if $ \bar\beta(j,0)+1\le R\le  \bar\beta(j,0)+2$.}
    \end{cases}
    $$
Finally, define
$$\ell_{j+1}(R):=(1+\rho_j(R))\ell_j(R) \qquad \text{for all $R\ge0$.}
$$
Clearly, $\ell_{j+1}$ satisfies \eqref{Nj}. 

\medskip \emph{Step 3. (Definition of  $\ell$ and $\Psi$)}
As anticipated above, we define $\ell$ as the pointwise limit of the increasing sequence $(\ell_j)$. Incidentally notice that, by construction, $\ell$ is continuous and $\ell(R)\geq \ell_{j}(R)$ for all $R\in[0,+\infty)$ and all $j\geq 1$. Furthermore, let  $\Psi:[0,+\infty)\to[0,+\infty)$ be any continuous strictly increasing function, satisfying
$$
{R\in [j-1,j]}\quad \Rightarrow \quad \Psi(R)\geq \bar\beta(j,0)+2 \quad \ \text{for all} \   j\geq 1.
$$ We show that $\ell$ and $\Psi$ have the required properties. Let  $z\in \R^n\setminus \C$ be given.

In order to prove (i), it suffices to observe that there exists an integer $j\ge 1$ such that $j-1< \d(z)\leq j$. Then,   
 \eqref{Rjest} implies that 
\begin{equation}\label{J_j}
J(z,\hat u)=J_j(z,\hat u)\leq L_j+\Phi(j)<+\infty.
\end{equation}

Let us now prove (ii). Assume   $\d(z)\leq r_1<1$. Then,   $T(\d(z))=0$, $\ell(\d(\hat x(t)))=\ell_1(\d(\hat x(t)))$ for all $t\geq 0$, and, arguing as in Step 2,  we obtain
\begin{align*}
    J(z,\hat u)&=\int_0^{+\infty}\left[ \ell_1(\d(\hat x(t))+l(\hat x(t),\hat u(t))\right]dt
    \leq\int_0^{+\infty} \bar \beta(\d(z),t)\, e^{-t}\,dt + \Phi(\d(z))\\
    &\le \bar \beta(\d(z),0)+\Phi(\d(z)) \qquad \text{for all $t\geq 0$.}
    \end{align*}
    
    To prove (iii), consider $u\in \mathcal M([0,+\infty),U)$  satisfying $J(z,u)\leq J(z,\hat u)$. Actually,  $u\in\U(z)$, so let   $(x^0,x,u)$ be the corresponding admissible triple from $(0,z)$. As above, let $j$ be the integer $\ge 1$ such that $j-1<\d(z)\le j$. By the properties of $\bar\beta$,  we have $\bar\beta(j,0)>j$, so that   $\d(z)<\bar\beta(j,0)$. In contradiction to claim  (iii),  suppose that there exists $t>0$ such that 
    $$
    \d(x(t))>\Psi(\d(z))>\bar\beta(j,0)+2.
    $$
     Then,  there exist $0<t_1<t_2<T_z(u)$ such that  
    \begin{equation}\label{t_1t_2}
    \begin{array}{l}
    \d(x(t_1))= \bar\beta(j,0)+1, \qquad \d(x(t_2))=\bar\beta(j,0)+2, \\[1.5ex]
    \bar\beta(j,0)+1\leq \d(x(t))\leq\bar\beta(j,0)+2\qquad \ \text{for all} \   t\in[t_1,t_2].
    \end{array}
    \end{equation}
    Therefore, in view of  \eqref{J_j},  \eqref{gamma}, and the definition of $(\ell_j)_j$,  we have
    \begin{align*}
       J(z,u)&=\int_0^{T_z(u)} \left[\ell(\d(x(t)))+l(x(t),u(t))\right]dt 
       \geq \int_0^{T_z(u)} \ell(\d(x(t)))dt\\
       &> \int_{t_1}^{t_2} \ell_{j+1}(\d(x(t)))dt 
      \geq (1+\rho_j(\bar\beta(j,0)+1))\int_{t_1}^{t_2} \ell_1 (\d(x(t)))dt\\
       &\geq (1+\rho_j(\bar\beta(j,0)+1))\varkappa_j=  L_j+\Phi(j)
       \geq J_j(z,\hat u)=J(z,\hat u).
    \end{align*}
    This provides the required contradiction and the proof of (iii) is complete.

  As in claim (iv),  let us now suppose that $u\in\U(z)$ is a control satisfying $J(z,u)<\alpha$, for some $\alpha>0$. Let  $(x^0,x,u)$ be the corresponding admissible triple from $(0,z)$, and let $j\ge1$ be the smallest integer  such that $L_j+\Phi(j)>\alpha$. We want to show  that 
$$
\d(z)<\Theta:=\bar\beta(j,0)+2. 
$$
Indeed, assume instead that $\d(z)\geq \bar\beta(j,0)+2$. If there exists a time $t\ge0$ such that  $\d(x(t))<c:=\bar\beta(j,0)+1$, then there are  $0<t_1<t_2<T_z(u)$ as in \eqref{t_1t_2}. Thus,  arguing as in the proof of claim (iii) we can deduce the inequality $J(z,u)> L_j+\Phi(j)>\alpha$, in contradiction with the hypothesis $J(z,u)<\alpha$.  If instead $\d(x(t,z,u))\geq c$ for all $t\leq T_z(u)$, then $T_z(u)=+\infty$ and we get the contradiction
    $$J(z,u)\geq \int_0^{+\infty} \ell_1(\d(x(t)))dt\geq \int_0^{+\infty} \ell_1(c)dt=+\infty. $$

  Let finally prove (v).  Let $\alpha>0$ and assume $\d(z)>\alpha$. For any  $u\in \mathcal M([0,+\infty),U)\setminus\U(z)$, the cost $J(z,u)=+\infty$, so (v) is trivially true. If $u\in\U(z)$,  by the last part of the proof of (iv) (for $c:=\alpha/2$) there exists a positive, finite time $t_2:=\inf\{t\in[0,T_z(u))\mid \d(x(t,u,z))\leq \alpha/2\}$. Let $0<t_1<t_2$ be a time such that $\alpha/2\le \d(x(t,u,z)\le\alpha$ for all $t\in[t_1,t_2]$. Then, in view of \eqref{gamma}, we have
$$J(z,u)\geq \int_{t_1}^{t_2} \ell_1(\d(x(t,u,z)))dt\geq \frac{\alpha}{2}\varkappa(\alpha/2,\alpha)=:\delta>0.$$
\end{proof}

Let $J$ be as in Lemma \ref{l2}. For any $z\in\overline{\R^n\setminus\C}$, we define the value function  
$$
V(z):=\inf_{u\in\M([0,+\infty),U)} J(z,u) \quad \text{for all} \  z\in \R^n\setminus\C, \qquad V(z):=0  \quad\text{for all $z\in\partial\C$.}
$$
 In the following two lemmas we show  that $V$ is a MRF for \eqref{Egen}-\eqref{Cgen}.  
 \begin{lemma}\label{l3} The function  $V: \overline{\R^n\setminus\T}\to[0,+\infty)$ has the following properties:
\begin{enumerate}[{\rm (i)}]
\item $\text{dom }V:=\{z\in \overline{\R^n\setminus\T} \mid V(z)<+\infty\}=\overline{\R^n\setminus\T}$;
\item  $V$ is positive definite;
\item  $V$ is proper;
\item $V$ is continuous.
\end{enumerate}
\end{lemma}
\begin{proof} Property (i) holds, because,   if $z\in\partial\C$, then $V(z)=0<+\infty$,  while, if $\d(z)>0$,  by Lemma \ref{l2},\,(i) we have  $V(z)\leq J(z,\hat u)<+\infty$. 

 In order to prove (ii),  observe that given $z\in\R^n\setminus\C$, then $\d(z)>\alpha$ for some $\alpha>0$. Hence, by Lemma \ref{l2},\,(v) there exists $\delta>0$, depending only on $\alpha$, such that $J(z,u)>\delta$ for all $u\in \mathcal M([0,+\infty),U)$. As a consequence,  $V(z)\geq \delta>0$, that is $V$ is positive outside the target. 
 
The function $V$ satisfies (iii) whenever, for all $\alpha>0$,  the sublevel set $E_\alpha:=\{z\in \overline{\R^n\setminus\T}\mid V(z)<\alpha\}$ is bounded. If $V(z)<\alpha$, then by definition there exists $u\in \U(z)$ such that $J(z,u)<\alpha$, as well. Hence, by Lemma \ref{l2},\,(iv) we deduce that $\d(z)<\Theta$ for some $\Theta>0$ (depending only on $\alpha$) and, consequently, the set  $E_\alpha$ is bounded, as $E_\alpha\subset \overline{B_\Theta(\T)\setminus\T}$.

 Let us finally prove (iv), i.e. the continuity of $V$. Fix $\varepsilon>0$ and let us first consider  $\bar z\in\partial \T$.    By virtue of Lemma \ref{l2},\,(ii),   for any $z$ such that $\d(z)<r_1$,  we have $V(z)\leq J(z,\hat u)\leq \hat\Phi(\d(z)):=\bar\beta(\d(z),0)+\Phi(\d(z))$, where, in particular,  $\hat\Phi$ is continuous, strictly increasing and equal to 0 at 0. Hence, choosing    
 \begin{equation}\label{delta_e}
 0<\delta_\varepsilon<\frac{1}{2} \left(r_1\land \hat\Phi^{-1}(\varepsilon)\right), 
\end{equation}
 we get  the continuity of $V$ at $\bar z$, as  
 \begin{equation}\label{V_e}
 |V(z)-V(\bar z)|=V(z)\leq J(z,\hat u)\le \hat\Phi(\d(z))<\varepsilon\quad \text{for all $z\in B_{2\delta_\varepsilon}(\bar z)$. }
\end{equation}
%

\noindent Assume now $\bar z\in\R^n\setminus\C$. {Setting  $\tilde\delta_{\bar z,\varepsilon}:=\delta_{\varepsilon}\land \frac{\d(\bar z)}{4}$, we have $B_{2\tilde\delta_{\bar z,\varepsilon}}(\bar z)\subset  \R^n\setminus\C$.}  We claim that  for any $z\in B_{\tilde\delta_{\bar z,\varepsilon}/2}(\{\bar z\})$  there exists an admissible triple  $(x^0,x,u)$ from $(0,z)$    such that 
  \begin{equation}\label{dis_dJ}
    \begin{array}{l}
    \displaystyle \d(x(t))\leq M_{\bar z}:=\max\left\{\Psi(\d(z))\mid z\in \overline{B_{r_1/4}(\{\bar z\})}\right\}\quad \text{for all $t\geq 0$,}
     \\[1.5ex]
    J(z,u)\le V(z)+\varepsilon.
    \end{array}
   \end{equation}
    Indeed, an $\varepsilon$-optimal triple  $(x^0,x,u)$ satisfying $J(z,u)\le V(z)+\varepsilon$ always exists, as $V(z)<+\infty$ by (i) above. 
Furthermore,  we can clearly assume  $J(z,u)\le J(z,\hat u)$, where $(\hat x^0,\hat x,\hat u)$ is the admissible triple from $(0,z)$ built  in Lemma \ref{l2}. But then the first of the inequalities above follows from Lemma \ref{l2},\,(iii).

 \noindent   Set 
 \begin{equation}\label{tildeT}
 \tilde T_{z}:=\inf\{t\ge0\mid \d(x(t))<\tilde\delta_{\bar z,\varepsilon}\}. 
   \end{equation}
{Note that $\tilde T_{z}>0$, since $\d(z)>\d(\bar z)-\frac{\tilde\delta_{\bar z,\varepsilon}}{2}>\frac{3}{2}\tilde\delta_{\bar z,\varepsilon}$.} Let $j=j_{\bar z,\varepsilon}$ be an integer $\ge1$    such that 
 $j\geq \d(\bar z)+\frac{\tilde\delta_{\bar z,\varepsilon}}{2}> \d(z)$. Then,  using  \eqref{J_j} (which is valid for every $j\ge \d(z)$,  in view of  \eqref{Rjest}) and  setting   $\displaystyle m_{\bar z,\varepsilon}:=\min_{R\in[\tilde\delta_{\bar z,\varepsilon},   M_{\bar z}]}\ell_1(R)>0$,  we get    
\begin{align*}
   L_j+\Phi(j)&\geq J(z,\hat u)\geq J(z, u) 
     \geq \int_0^{\tilde T_{z}}\ell_1(\d(x(t)))dt \geq \tilde T_{z} m_{\bar z,\varepsilon}.
\end{align*}
Thus, there exists an uniform upper bound for the times $\tilde T_z$. Precisely,   we have 
$$
\tilde T_z\le T_{\bar z,\varepsilon}:=\frac{ L_{j_{\bar z,\varepsilon}}+\Phi(j_{\bar z,\varepsilon})}{m_{\bar z,\varepsilon}}   \qquad \text{for all  $z\in B_{\tilde\delta_{\bar z,\varepsilon}/2}(\{\bar z\})$.}
$$
To prove the continuity of $V$ at $\bar z$,   consider  arbitrary points $z_1,z_2\in B_{\tilde\delta_{\bar z,\varepsilon}/2}(\{\bar z\})$ and suppose, for instance,  $V(z_1)\leq V(z_2)$. Let $(x_1^0,x_1,u_1)$ be an admissible triple from $(0,z_1)$ satisfying \eqref{dis_dJ} (for $z=z_1$). In particular, this implies that $x_1(t)$ lies in a compact set $\K$ depending only on $\bar z$  for all $t\ge0$. Hence,  if $L_{\bar z}$ denotes the Lipschitz constant in $x$  of the dynamics function $f$ on the compact set $\overline{B_1(\K)}$, by a standard cut-off technique we can derive that the  trajectory $x(\cdot, u_1,z_2)$ is defined for all $t\in[0, T_{\bar z,\varepsilon}]$ and satisfies  
$$
\displaystyle\sup_{t\in[0, T_{\bar z,\varepsilon}]} |x_1(t)-x(t, u_1,z_2)|\le|z_1-z_2|\,e^{L_{\bar z}\,T_{\bar z,\varepsilon}},
$$
as soon as $|z_1-z_2|< e^{-L_{\bar z}\,T_{\bar z,\varepsilon}}$.  Actually, from this inequality it also follows that, setting   
$$
\bar \delta_{\bar z,\varepsilon}:=\tilde\delta_{\bar z,\varepsilon}\land \left(\frac{\delta_\varepsilon}{2}\land 1\right)e^{-L_{\bar z}\,T_{\bar z,\varepsilon}}
$$
($\delta_\varepsilon$ as in \eqref{delta_e}), and assuming $z_1$, $z_2\in  B_{\bar\delta_{\bar z,\varepsilon}/2}(\{\bar z\})$, we have that ($|z_1-z_2|< \bar \delta_{\bar z,\varepsilon}$ and) $|x(\tilde T_{z_1},u_1,z_2)-x_1(  \tilde T_{z_1})|\le \frac{\delta_\varepsilon}{2}$ ($\tilde T_{z_1}$ is as in \eqref{tildeT}, for $x=x_1$). Since $\d(x_1(\tilde T_{z_1}))\le \delta_\varepsilon$, this implies that  $\bar z_2:=x(\tilde T_{z_1},u_1,z_2)\in B_{2\delta_\varepsilon}(\bar z)$.  At this point, if $\hat u_2\in\U(z_2)$ denotes an admissible control from $(0,\bar z_2)$ as in  Lemma \ref{l2}, the last part of \eqref{V_e} implies that $J(\bar z_2,\hat u_2)<\varepsilon$.  Therefore, the control $u_2$ given by 
$$u_2(t):=\begin{cases} u_1(t) \quad &t\in [0,\tilde T_{z_1}]\\
 \hat u_2(t-\tilde T_{z_1}) &t\in(\tilde T_{z_1}, +\infty)\end{cases}$$
 belongs to $\U(z_2)$,  $x_2(t):=x(t, u_2,z_2)$ belongs to $\overline{B_1(\K)}$ for all $t\in[0,\tilde T_{z_1}]$, and denoting with $\omega_{\bar z}$   the modulus of continuity of $\overline{B_1(\K)}\ni x\mapsto\ell(\d(x))+l(x,u)$ (uniform w.r.t. the control, because of the assumptions on $l$), we  finally obtain
%
\begin{align*}
    0&\le  V(z_2) -V(z_1) \leq J(z_2,u_2)-J(z_1,u_1)+\varepsilon   \\
      &\leq \int_{0}^{\tilde T_{z_1}}\left[|\ell(\d(x_2(t))) -\ell(\d(x_1(t)))|\,dt+ |l(x_2(t),u_1(t)) 
-l(x_1(t),u_1(t))|\right]dt \\
& \qquad +J(\bar z_2, \hat u_2)+\varepsilon \le    T_{\bar z,\varepsilon}\, \omega_{\bar z} (\bar \delta_{\bar z,\varepsilon})+2\varepsilon\le 3\varepsilon,
\end{align*}
where the last inequality holds by replacing $\bar \delta_{\bar z,\varepsilon}$ with  $\bar \delta_{\bar z,\varepsilon}\land \omega_{\bar z}^{-1}\left(\frac{\varepsilon}{T_{\bar z,\varepsilon}}\right)$.  The continuity of $V$ at $\bar z$ hence follows by the arbitrariness of $\varepsilon$.

\end{proof}

\begin{lemma}\label{l4}
The value function $V$ satisfies the decrease condition \eqref{MRH}.
\end{lemma}
\begin{proof} We divide the proof in three steps.    
\vsm
{\it Step 1.}  Let us  first show that, if  $V$ is a viscosity supersolution of the following Hamilton-Jacobi-Bellman equation
\begin{equation}\label{sup_sol}
\max_{u\in U} \{-\langle DV(z),f(z,u)\rangle-l(z,u)\}=\ell(\d(z)) \qquad \text{for all} \ z\in\R^n\setminus\T,
\end{equation}
then  it satisfies the decrease condition \eqref{MRH},  characterizing MRFs.  
Indeed, \eqref{sup_sol} implies that (see e.g.  the survey paper \cite{C10})
$$
\max_{u\in U} \{-\langle p,f(z,u)\rangle-l(z,u)\}\ge \ell(\d(z)) \qquad \text{for all} \ z\in\R^n\setminus\T, \quad \text{for all} \ p\in  \partial_P V(z),
$$
so that, for $H$ defined as in \eqref{Ham} and  $p_0\equiv 1$, one has
$$
H(z,1,\partial_P V(z))\leq - \ell(\d(z))\qquad \text{for all} \ z\in\R^n\setminus\T, \quad \text{for all} \ p\in  \partial_P V(z), 
$$
Setting $\gamma:=\ell\circ d_{V^+}^{-1}$ and recalling that $L$ is increasing as well as $d_{V^+}$, by \eqref{Ldis}  we finally obtain that $V$ satisfies condition \eqref{MRH} for $p_0\equiv 1$ and such a $\gamma$, namely
$$
H(z,1,\partial_P V(x))\leq - \gamma(V(z))\quad \ \text{for all} \   z\in \R^n\setminus\T.
$$
Thus, the next two steps will be devoted to prove that $V$ is a viscosity supersolution of \eqref{sup_sol}.  This proof is not completely standard, because we do not have the usual growth hypotheses on  $f$, $\ell$ and $l$ (see e.g. \cite[Ch. III]{BCD}). Actually, these assumptions  can be  avoided here thanks to the results in Lemma \ref{l2}.  
\vsm
{\it Step 2.}  Let us  show that, for every $T>0$ and every $z\in\R^n\setminus\T$, one has 
 \begin{equation}\label{DPP}
 V(z)\ge 
 \displaystyle \inf_{u\in\hat\U(z)}\left\{ \int_0^{T_z(u)\land T}[\ell(\d(x(t)))+l(x(t),u(t))]dt+V(x(T_z(u)\land T))\right\},
 \end{equation}
where $x:=x(\cdot\,, u,z)$, $\hat\U(z):=\{u\in\hat\U(z)\mid \ J(z,u)\le J(z,\hat u)\}$,   and $\hat u$ is as in Lemma \ref{l2}.  In view of Lemma \ref{l2},(i), the set  $\hat\U(z)\ne\emptyset$   and  
 $$
V(z)=\inf_{u\in\hat\U(z)} \int_0^{T_z(u)}[\ell(\d(x(t)))+l(x(t),u(t))]dt<+\infty.
$$
Let us refer to the right-hand side of \eqref{DPP} as $v_T(z)$. Given $u\in\hat\U(z)$, if $T_z(u)\le T$ we have $V(x(T_z(u)\land T))=0$ and $J(x,u)\ge v_T(z)$. If instead $T_z(u)> T$, in view of the definition of $v_T$,  we get
$$
\begin{array}{l}
 \displaystyle J(z,u)=\int_0^{T}[\ell(\d(x(t)))+l(x(t),u(t))]dt+\int_T^{T_z(u)}[\ell(\d(x(t)))+l(x(t),u(t))]dt \\
 \displaystyle =\int_0^{T}[\ell(\d(x(t)))+l(x(t),u(t))]dt+\int_0^{T_{x_T(0)}(u_T)}[\ell(\d(x_T(t)))+l(x_T(t),u_T(t))]dt \\
  \displaystyle \ge \int_0^{T}[\ell(\d(x(t)))+l(x(t),u(t))]dt+V(x(T))\ge v_T(z),
\end{array}
$$
where $x_T(t):=x(t+T)$, $u_T(t):=u(t+T)$, and $T_{x_T(0)}(u_T)=T_z(u)-T$. Therefore, $V(z)=\inf_{u\in\hat\U(z)}J(z,u)\ge v_T(z)$, that is, the relation \eqref{DPP} is proven. 

\vsm
{\it Step 3.} Let us now deduce from  \eqref{DPP} that $V$ is a viscosity supersolution of \eqref{sup_sol}. 
To this aim, fixed $z\in\R^n\setminus\T$, we preliminarily  observe that, in view of Lemma \ref{l2},(iii),   every admissible trajectory-control pair $(x,u)$ with  $u\in\hat\U(z)$,  satisfies
 $$
\d(x(t))\le \Psi(\d(z))=:R_z \qquad \text{for all} \  t\ge0.
$$
Furthermore,  on the compact set $\overline{B(\T, R_z)\setminus\T}$, depending only on $z$, the functions  $f$ and $l$ are Lipschitz continuous in $x$, uniformly w.r.t. the control, and we can fix a modulus of continuity for $\ell$ and a bound $M_z>0$ for $|f|$, $\ell$, and $l$. 
Hence,  choosing e.g.     $\bar T_z=\frac{\d(z)}{2 M_z}$,   for any  $u\in\hat\U(z)$   the trajectory $x(\cdot\,, u,  z)$ is defined on   $[0,\bar T_z]$ and satisfies $0<  \d( x(t, u,  z) )\le R_z$ for all $t\in [0,\bar T_z]$.  For arbitrary $\varepsilon>0$ and $0<T<\bar T_z$,   \eqref{DPP} implies that there exists some $\bar u=\bar u_{\varepsilon,T}\in\hat\U(z)$, such that 
$$
\int_0^{T}[\ell(\d(\bar x(t)))+l(\bar x(t),\bar u(t))]dt+V(\bar x(T))\le V(z)+\varepsilon T,
$$
where $\bar x:=x(\cdot,\,\bar u,z)$.  In view of the above considerations, from now on,   taken a test function  $\varphi\in C^1(\R^n)$  such that $\varphi(z)=V(z)$ and $\varphi(\tilde z)\le V(\tilde z)$ for all $\tilde z\in B(z,r)$,  for some $r>0$, the proof that $V$ is a viscosity supersolution of  \eqref{sup_sol} proceeds as usual, hence we omit it   (see e.g. \cite[Prop. III, 2.8]{BCD}).  

\end{proof}
 Surveying the results on  $V$ in Lemmas  \ref{l3} and  \ref{l4}, we see that the proof of implication (i) $\Longrightarrow$ (ii) is concluded. \qed
 
 \section{Proof of implication {\rm (ii)} $\Longrightarrow$ {\rm (i)}}\label{s4}   
 The proof that the existence of a MRF  implies GAC to $\T$ with regulated cost relies on the following result, establishing a    super-optimality principle satisfied by any MRF.  
\begin{proposition}\label{Th_supsol}
 Let $W:\overline{\R^n\setminus\C}\to[0,+\infty)$ be a continuous MRF for \eqref{Egen}-\eqref{Cgen}  for some continuous and increasing  function  $p_0:(0,+\infty)\to[0,1]$ and some continuous and strictly increasing  function $\gamma:(0,+\infty)\to(0,+\infty)$.   Then, for any $z\in\R^n\setminus\C$, we have
\begin{equation}\label{SOP}
W(z)\ge 
 \inf_{u\in \U(z)}\sup_{0\le T< T_z(u)}\left\{\begin{array}{r}  \displaystyle\int_0^{ T}\left[p_0(W(x(t)))l(x(t),u(t))+\gamma(W(x(t)))\right]dt \\
 +W(x(T))
 \end{array}\right\},
\end{equation}
where $x(\cdot):=x(\cdot\,,u,z)$.
\end{proposition}
The proof of this proposition relies on the definite positiveness and properness of $W$ coupled with an extension of results which are already known only under more restrictive assumptions than ours  (see e.g. \cite[Thm. 2.40]{BCD} and  \cite[Thm. 3.3]{GS04}), and it  will be given at the end of the section. 

\vsm  Let $W:\overline{\R^n\setminus\C}\to[0,+\infty)$ be a continuous MRF for \eqref{Egen}-\eqref{Cgen}  for some   functions $p_0:(0,+\infty)\to[0,1]$ and $\gamma:(0,+\infty)\to(0,+\infty)$ as in Def. \ref{defMRF}. Moreover, assume that $p_0$ satisfies the integrability  condition (IC).  Fix $z\in\R^n\setminus\C$.  From \eqref{SOP} it follows that, for  $\varepsilon_1:=\frac{W(z)}{2}>0$ there exists some admissible control $u_1\in\U(z)$ such that, for any $0\le T<T_z(u_1)$, we have
\begin{align}\label{OPTe}
\displaystyle\int_0^{ T}[p_0(W(x_1(t)))l(x_1(t),u_1(t)) +\gamma(W(x_1(t)))]dt &+W(x_1(T)) \nonumber \\
 &\le W(z)+\frac{W(z)}{2}, 
\end{align}
where $x_1(\cdot):=x(\cdot\,,u_1,z)$. We set
$$
t_1:=\inf\left\{t\in[0,T_z(u_1))\mid \ W(x_1(t))\le  \frac{W(z)}{2}\right\}, \qquad z_1:=x_1(t_1).
$$
Clearly, $t_1$ is $>0$ and  is actually a minimum. Hence, \eqref{OPTe} implies that
\begin{equation}\label{rel1}
\left\{\begin{array}{l}\displaystyle\frac{1}{2}\,W(z)\le W(x_1(t))\le  \frac{3}{2}\,W(z) \qquad \text{for all $t\in[0,t_1]$,} \\[1.5ex]
\displaystyle W(z_1)=W(x_1(t_1))=\frac{1}{2}\,W(z),
\end{array}\right.
\end{equation} 
so that 
\begin{align*}
\displaystyle p_0(W(z_1))\,\int_0^{ t_1}&l(x_1(t),u_1(t))\,dt\le \int_0^{ t_1}p_0(W(x_1(t)))\,l(x_1(t),u_1(t))\,dt \\
&\le  W(z)+\frac{W(z)}{2}- W(z_1)=W(z)=2[W(z)-W(z_1)],
\end{align*}
which yields  the cost bound
\begin{equation}\label{rel2}
\int_0^{ t_1}l(x_1(t),u_1(t))\,dt\le 2\frac{W(z)-W(z_1)}{p_0(W(z_1))}.
\end{equation} 
From \eqref{OPTe}, using  the functions $d_{W^-}$, $d_{W^+}$ introduced  in \eqref{Ldis}, we also obtain  
{\small$$
\gamma\left(\frac{1}{2}\,d_{W^-}(\d(z))\right)\,t_1\le\gamma(W(z_1))\,t_1\le\int_0^{t_1}\gamma(W(x_1(t)))\,dt\le \frac{3}{2}\,W(z)\le  \frac{3}{2}\,d_{W^+}(\d(z)).
$$}
Define now the continuous   function  $\bar T_1:(0,+\infty)\to(0,+\infty)$, given by 
$$
\bar T_1(R):=\frac{3d_{W^+}(R)}{2\gamma\left(\frac{1}{2}\,d_{W^-}(R)\right)} \qquad \text{for all $R>0$.}
$$
 Hence, the latter inequality   yields the following uniform time bound 
\begin{equation}\label{rel3}
t_1\le \bar T_1(\d(z)).
\end{equation} 
Starting from $z_1$ and choosing $\varepsilon_2:=\frac{W(z_1)}{2}=  \frac{W(z)}{4}$, arguing as above we can deduce  from \eqref{OPTe}  the existence of a control $u_2\in\U(z_1)$ and a time $t_2>0$, such that, denoting by $x_2$ the trajectory $x(\cdot\,,u_2,z_1)$ and setting $z_2:=x_2(t_2)$, we get relations \eqref{rel1}-\eqref{rel3}  with $z$, $z_1$, $u_1$, $x_1$,  and $t_1$ replaced by $z_1$, $z_2$, $u_2$, $x_2$,  and $t_2$, respectively.  Set  $z_0:=z$.  In a recursive way,  for any integer $N\ge1$, we can thus choose $\varepsilon_N:=  \frac{W(z)}{2^N}$ and construct  $z_N$, $u_N$, $x_N$, and $t_N>0$, such that  $u_N\in\U(z_{N-1})$, and $x_N(\cdot):=x(\cdot\,,u_N,z_{N-1})$,   $z_N:=x_N(t_N)$,  satisfy 
\begin{equation}\label{rel1N}
\left\{\begin{array}{l}\displaystyle\frac{1}{2}\,W(z_{N-1})\le W(x_N(t))\le  \frac{3}{2}\,W(z_{N-1}) \qquad \text{for all $t\in[0,t_N]$,} \\[1.5ex]
\displaystyle W(z_N)=\frac{1}{2}\,W(z_{N-1})=\frac{1}{2^N}\,W(z),
\end{array}\right.
\end{equation} 
\begin{equation}\label{rel2N}
\int_0^{ t_N}l(x_N(t),u_N(t))\,dt\le 2\,\frac{W(z_{N-1})-W(z_N)}{p_0(W(z_N))}=4\,\frac{W(z_{N})-W(z_{N+1})}{p_0(W(z_N))},
\end{equation} 
and
\begin{equation}\label{rel3N}
t_N\le \bar T_N(\d(z)),
\end{equation} 
where 
$$
\bar T_N(R):=\frac{3d_{W^+}(R)}{2^N\gamma\left(\frac{1}{2^N}\,d_{W^-}(R)\right)} \qquad \text{for all $R>0$.}
$$
Set now $T_0:=0$, $T_N:=\sum_{j=1}^Nt_j$, and  $T_{\infty}:=\sum_{j=1}^{+\infty}  t_{j}$ 
and, for every integer $N\ge1$,  define the control $u\in\M([0,+\infty),U)$ given by
\begin{align*}\label{P_est}
u(t)&:=u_N(t-T_{N-1})\qquad \text{for all }t\in[T_{N-1},\ T_{N}), \\
u(t)&:=w  \qquad \text{for all }t\ge T_{\infty}, \quad \text{if $T_{\infty}<+\infty$,}
\end{align*}
(for $w\in U$ arbitrary). Recalling that $W$ is proper and positive definite, from \eqref{rel1N} it  is easy to deduce that   $\lim_{t\to T_\infty^-}\d(x(t,u,z))=0$, so    $u\in\U(z)$.  Let $(x^0,x,u)$ be the corresponding admissible triple from $(0,z)$ (defined on $[0,+\infty)$). In view of  \eqref{rel2N} and recalling that $1/p_0$ is decreasing, the cost $x^0$ satisfies
 \begin{align}
x^0(t)\le \int_0^{T_\infty}l(x(t),u(t))\,dt&\le\sum_{N=1}^{+\infty}4\,\frac{W(z_{N})-W(z_{N+1})}{p_0(W(z_N))} \\
&\le4\,\int_0^{W(z)/2}\frac{dv}{p_0(v)}=\bar W(z) \qquad \text{for every $t\ge0$,}\nonumber
\end{align}
as soon as we set $\bar W(z):=4P(W(z)/2)$ for all $z\in \overline{\R^n\setminus\C}$ ($P$ as in \eqref{P}).
Notice that this function $\bar W$ is continuous, proper and positive definite, by  the integrability assumption (IC). So, the triple $(x^0,x,u)$ satisfies the cost bound condition \eqref{DWreg} with regulation function $\bar W$.  To conclude the proof that control system \eqref{Egen} with cost \eqref{Cgen} is GAC to $\C$ with regulated cost, it remains only to prove the existence of a descent rate $\beta$, such that 
\begin{equation}\label{beta_proof}
\d(x(t))\le\beta(\d(z),t)  \qquad \text{for every $t\ge0$.}
\end{equation}
First of all, we claim that there exist a strictly increasing, unbounded, continuous function   $\Gamma:\R_{\ge0}\to\R_{\ge0}$   with $\Gamma(0)=0$,  and  a function $\mathbf{T}:\R_{>0}^2\to \R_{>0}$, such that,  for any   $0<r<R$,    for every $z\in \R^n\setminus\C$ with $\d(z)\leq R$, the trajectory $x$ from $z$ considered above  satisfies the following conditions:
	$$
	\begin{array}{lllll}
		
		&{\rm (a)} &\d(x(t)) \leq {\bf \Gamma}(R)\qquad &\ \text{for all} \   t \geq 0,\\[1.5ex]
		&{\rm (b)} &\d(x(t)) \leq r \qquad &\ \text{for all} \   t \geq \mathbf{T}(R,r).
	\end{array}
	$$
 Condition (a) follows from  \eqref{rel1N}, because by  \eqref{Ldis}  we have
$$
\d(x(t))\le d^{-1}_{W^-}(W(x(t)))\le  d^{-1}_{W^-}\left(\frac{3}{2}\,d_{W^+}(\d(z))\right)\le
\Gamma(R) \quad \text{for all $t\ge0$,}
$$
as soon as we choose $\Gamma(R):=d^{-1}_{W^-}\left(\frac{3}{2}\,d_{W^+}(R)\right)$, $R\ge0$. This $\Gamma$ has all the required properties in view of the properties of $d_{W^-}$ and $d_{W^+}$. In order to derive (b), we observe that  \eqref{rel1N}  implies  
$$
\d(x(t))\le d^{-1}_{W^-}(W(x(t)))\le  d^{-1}_{W^-}\left(\frac{3}{2^N}\,d_{W^+}(R)\right) \quad\text{ for all $t\ge T_N$,}
$$
so,  if  $N(R,r)$ is the smallest integer  $\ge \log_2\left(3\frac{  d_{W^+}(R)}{d_{W^-}(r)}\right)$, we get
$$
\d(x(t))\le r \qquad\text{ for all $t\ge T_{N(R,r)}$.}
$$
The time $T_{N(R,r)}$  depends   on $z$,  but, by   \eqref{rel3N},  the value
  $$
\mathbf{T}(R,r):=\sum_{j=1}^{N(R,r)}\bar T_{j}(R)
$$
is a uniform upper bound for $T_{N(R,r)}$. Hence, also condition (b) is valid. Now, arguing as in \cite[Sec.\,5]{S99},  for any $R>0$  let us introduce a strictly increasing, diverging sequence of positive times $(t_j)_{j\ge1}$ depending on $R$, such that
$$
t_j\ge \mathbf{T}\left(R,\frac{R}{j+1}\right)
$$
and define the function $b:[0,+\infty)\times [0,+\infty)\to[0,+\infty)$, given by
$$
b(R,t):=\begin{cases} 
\Gamma(R)\qquad\text{for all $t\in[0,t_1)$,} \\
\dfrac{R}{j+1} \qquad\text{for all $t\in[t_j, t_{j+1})$, \, $j\ge1$.}
\end{cases}
$$
From (a) and (b) it follows that 
$$
\d(x(t))\le b(\d(z),t)\qquad\text{for all $t\ge0$.}
$$
As already noticed in the proof of Lemma \ref{l1}, Step 3, it is actually a routine exercise to find a $\K\L$ function $\beta\ge b$. Therefore, the proof of  implication {\rm (ii)} $\Longrightarrow$ {\rm (i)} is thus complete. \qed

 \subsection{A Comparison Principle and the proof of Proposition  \ref{Th_supsol}}
In the proof of Proposition  \ref{Th_supsol},  we  will use the slightly modified  version below of the classical Comparison Principle for the infinite horizon problem. In particular,  in the known results it is basically necessary to assume the  unilateral Lipschitz continuity hypothesis (i) below\footnote{An alternative hypothesis to (i)  is local Lipschitz continuity and at most linear growth in the state, uniformly w.r.t. the control.}   (see e.g.  the comments after \cite[III.\,Thm.\,2.12]{BCD}). In the following lemma we show that, when the state has two components and the dynamics  $F$ depends on only one of them,   $x$-Lipschitz continuity of the $F$-component that  corresponds to the missing variable is not necessary.  
   \begin{lemma}[Comparison Principle]\label{LCP} 
Let $U\subset\R^m$, $A\subset\R^{m'}$ be  compact control sets (not both empty), let $L:\R^n\times\R^{n'}\times U\times A\to \R$ be a  continuous running cost,  and let $F_1:\R^n\times U\times A\to \R^n $,  $F_2:\R^n\times U\times A\to  \R^{n'}$ be continuous dynamics components, such that 
\begin{enumerate}[{\rm (i)}]
\item for some $C>0$,   $\langle F_1(x,u,a)-F_1(x',u,a),x-x'\rangle\le C|x-x'|^2$ for all $x$, $x'\in\R^n$ and $u\in U$, $a\in A$;
\item for some $\bar K>0$ ,  $|F_2(x,u,a)|\le \bar K(1+|x|)$ for all $(x,u,a)\in\R^n\times U\times A$ and $x\mapsto F_2(x,u,a)$  is uniformly continuous, uniformly w.r.t. the controls;  
\item $L$ is bounded and  $(x,z)\mapsto L(x,z,u,a)$  is uniformly continuous, uniformly w.r.t. the controls.   
\end{enumerate}
 Let $H:\R^{n}\times\R^{n'}\times\R^{n}\times\R^{n'}\to \R$ be the Hamiltonian defined as
$$
H(x, z, p_1,p_2):=\min_{a\in A}\max_{u\in U}\Big\{-\langle (p_1,p_2),(F_1(x,u,a), F_2(x,u,a))\rangle -L(x,z,u,a)\Big\}.
$$
Given $\sigma>0$, if $v_1$, $v_2:\R^n\times\R^{n'}\to\R$, bounded and continuous, are, respectively, a viscosity   sub- and supersolution of
$$
\sigma\, v(x,z)+H(x, z, D v(x,z))=0 \qquad \text{for all }(x,z)\in\R^n\times\R^{n'},
$$
then $v_1(x,z)\le v_2(x,z)$ for all $(x,z)\in\R^n\times\R^{n'}$.
\end{lemma}

\begin{proof} The proof is a careful adaptation of the proof of \cite[III.\,Thm.\,2.12]{BCD}, which we give in detail for the sakes of clarity and selfconsistency. If we divide the equation by $\sigma>0$, the functions $F_1/\sigma$, $F_2/\sigma$ and $L/\sigma$ satisfy the same structural hypotheses as above. Hence, we can assume $\sigma=1$ without loss of generality. In view of \cite[III.\,Rem.\,2.13]{BCD} we can also assume $v_1$ and $v_2$ uniformly continuous.

Set $\langle x \rangle:=(1+|x|^2)^{\frac{1}{2}}$. Take the map $\Phi:\R^{n}\times \R^{n'}\times\R^{n}\times \R^{n'}\to\R$, given by
$$
\begin{array}{l} 
\displaystyle \Phi(x,z,y,w):=v_1(x,z)-v_2(y,w)-\frac{|x-y|^2}{2\varepsilon} -\frac{|z-w|^2}{2\rho} \\[1.5ex]
\qquad\qquad\qquad\qquad\qquad\qquad\qquad\qquad-\beta(\langle x\rangle^N+\langle z\rangle^N+\langle y\rangle^N +\langle w\rangle^N)
\end{array}
$$
where $\varepsilon$, $\rho$,  $\beta$, $N$  are positive parameters to be chosen conveniently.  

Suppose by contradiction that there exist  $\delta>0$ and  $(\tilde x,\tilde z)\in\R^{n}\times \R^{n'}$ such that $v_1(\tilde x,\tilde z)-v_2(\tilde x,\tilde z)=\delta$. Choose $\beta>0$ such that $\beta\langle \tilde x \rangle\le\delta/8$ and $\beta\langle \tilde z \rangle\le\delta/8$, so that for all $0<N\le 1$ we have $2\beta(\langle\tilde x \rangle^N+\langle\tilde z \rangle^N)\le2\left(\frac{\delta}{8}+\frac{\delta}{8}\right)=\frac{\delta}{2}$, and
$$
\frac{\delta}{2}\le \delta- \frac{\delta}{2}\le v_1(\tilde x,\tilde z)-v_2(\tilde x,\tilde z)-2\beta(\langle\tilde x \rangle^N+\langle\tilde z \rangle^N)=\Phi(\tilde x,\tilde z,\tilde x,\tilde z)\le\sup\Phi.
$$
Since $\Phi$ is continuous and tends to $-\infty$ as $|x|+|z|+|y|+|w|\to+\infty$, there exists a maximum point
$(\bar x,\bar z,\bar y,\bar w)$, for which, in particular,   we get
\begin{equation}\label{Lf1}
0<\frac{\delta}{2}\le\Phi(\bar x,\bar z,\bar y,\bar w)=\sup\Phi.
\end{equation}
The obvious inequality $\Phi(\bar x,\bar z,\bar x,\bar z)+\Phi(\bar y,\bar w,\bar y,\bar w)\le 2\Phi(\bar x,\bar z,\bar y,\bar w)$,  
yields
$$
\frac{|\bar x-\bar y|^2}{\varepsilon}+\frac{|\bar z-\bar w|^2}{\rho}\le v_1(\bar x,\bar z)-v_1(\bar y,\bar w)+v_2(\bar x,\bar z)-v_2(\bar y,\bar w),
$$
so the boundedness of $v_1$ and $v_2$ implies that 
\begin{equation}\label{Lf2}
|\bar x-\bar y|\le c \sqrt{\varepsilon}, \qquad |\bar z-\bar w|\le c\sqrt{\rho},
\end{equation}
for some $c>0$. From  the inequality $\Phi(\bar x,\bar z,\bar x,\bar w)+\Phi(\bar y,\bar z,\bar y,\bar w)\le 2\Phi(\bar x,\bar z,\bar y,\bar w)$ and  thanks to the uniform continuity of $v_1$, $v_2$, we can thus derive that  
\begin{equation}\label{Lf3}
\frac{|\bar x-\bar y|^2}{\varepsilon}\le v_1(\bar x,\bar z)-v_1(\bar y,\bar z)+v_2(\bar x,\bar w)-v_2(\bar y,\bar w) \le \omega(|\bar x-\bar y|)\le \omega(c\sqrt{\varepsilon}),
\end{equation}
for some modulus of continuity $\omega$.  

Consider now the $C^1$ test functions
$$
\begin{array}{l}
\displaystyle\varphi(x,z):= v_2(\bar y,\bar w)+\frac{|x-\bar y|^2}{2\varepsilon} +\frac{|z-\bar w|^2}{2\rho} +\beta(\langle x\rangle^N+\langle  z\rangle^N+\langle \bar y\rangle^N +\langle \bar w\rangle^N), \\[1.5ex]
\displaystyle\psi(y,w):= v_1(\bar x,\bar z)-\frac{|\bar x-  y|^2}{2\varepsilon} -\frac{|\bar z-  w|^2}{2\rho} -\beta(\langle \bar x\rangle^N+\langle \bar z\rangle^N+\langle y\rangle^N +\langle   w\rangle^N).
\end{array}
$$
By definition of $(\bar x,\bar z,\bar y,\bar w)$, the function $v_1(x,z)-\varphi(x,z)=\Phi(x, z,\bar y,\bar w)$ obtains its maximum  at 
$(\bar x,\bar z)$, while   $v_2(y,w)-\psi(y,w)=-\Phi(\bar x, \bar z, y, w)$ obtains its minimum  at 
$(\bar y,\bar w)$.  As it is easy to see, we have
$$
\begin{array}{l}
\displaystyle D\varphi(\bar x,\bar z)=(D_x\varphi,D_z\varphi)(\bar x,\bar z)=\left(\frac{\bar x-\bar y}{\varepsilon}+\gamma_1\bar x,\frac{\bar z-\bar w}{\rho}+\gamma_2\bar z\right),  \\[1.5ex]
\displaystyle D\psi(\bar y,\bar w)=(D_x\psi,D_z\psi)(\bar y,\bar w)=\left(\frac{\bar x-\bar y}{\varepsilon}+\tau_1\bar y,\frac{\bar z-\bar w}{\rho}+\tau_2\bar w\right),
\end{array}
$$
if $\gamma_1:=\beta N \langle \bar  x\rangle^{N-2}$,  $\gamma_2:=\beta N \langle \bar  z\rangle^{N-2}$,  $\tau_1:=\beta N \langle \bar  y\rangle^{N-2}$,   $\tau_2:=\beta N \langle \bar  w\rangle^{N-2}$. The definition of viscosity sub- and supersolution yields 
$$
v_1(\bar x,\bar z)+H\left(\bar x,\bar z, D\varphi(\bar x,\bar z)\right)\le0\le v_2(\bar y,\bar w)+H\left(\bar y,\bar w, D\psi(\bar y,\bar w)\right),
$$
so that,  for some $u\in U$ and $a\in A$, we have
\begin{align*}
v_1&(\bar x,\bar z)- v_2(\bar y,\bar w) \\
&\le H\left(\bar y, \bar w, \frac{\bar x-\bar y}{\varepsilon}+\tau_1\bar y,\frac{\bar z-\bar w}{\rho}+\tau_2\bar w\right)-H\left(\bar x,\bar z, \frac{\bar x-\bar y}{\varepsilon}+\gamma_1\bar x,\frac{\bar z-\bar w}{\rho}+\gamma_2\bar z\right) \\
&\le \left\langle \frac{\bar x-\bar y}{\varepsilon}+\gamma_1\bar x, F_1(\bar x,u,a)\right\rangle+ \left\langle \frac{\bar z-\bar w}{\rho}+\gamma_2\bar z, F_2(\bar x,u,a)\right\rangle +L(\bar x,\bar z, u,a) \\
&\qquad-\left\langle \frac{\bar x-\bar y}{\varepsilon}+\tau_1\bar y, F_1(\bar y,u,a)\right\rangle+ \left\langle \frac{\bar z-\bar w}{\rho}+\tau_2\bar w, F_2(\bar y,u,a)\right\rangle -L(\bar y,\bar w, u,a).
\end{align*}
Hence, by  standard calculations (see also \cite[Lemma 2.11]{BCD}), using the definitions of $\gamma_1$, $\gamma_2$, $\tau_1$, and $\tau_2$,  \eqref{Lf2} and \eqref{Lf3}, we get
 \begin{align*}
&v_1(\bar x, \,\bar z)- v_2(\bar y,\bar w) \le C\frac{|\bar x-\bar y|^2}{\varepsilon}+\frac{|\bar z-\bar w|}{\rho}\,\omega_2(|\bar x-\bar y|)+\omega_L(|\bar x-\bar y|+|\bar z-\bar w|)\\
&\qquad\qquad\qquad+\gamma_1K(1+|\bar x|^2)+\tau_1K(1+|\bar y|^2)+\gamma_2K(1+|\bar z|^2)+\tau_2K(1+|\bar w|^2) \\
&\le C\omega(c\sqrt{\varepsilon})+\frac{c\,\omega_2(c\sqrt{\varepsilon})}{\sqrt{\rho}}+\omega_L(c\sqrt{\varepsilon}+c\sqrt{\rho})  
 + K\beta N( \langle \bar  x\rangle^N+\langle \bar  z\rangle^N+\langle \bar  y\rangle^N+\langle \bar  w\rangle^N),
\end{align*}
where $C$ is as in (i) above,  $K$ is a suitable positive constant (depending on $C$ and $\bar K$, $\bar K$ as in (ii)),  and $\omega_2$, $\omega_L$ are  the moduli of continuity of $F_2$ and $L$, respectively.
At this point, by choosing $N:=1\land\frac{1}{K}$,  $\rho:=\omega_2(c\sqrt{\varepsilon})$,  and using \eqref{Lf1},  we obtain 
\begin{align*}
 \frac{\delta}{2}&\le\Phi(\bar x,\bar z,\bar y,\bar w)\le v_1(\bar x,\bar z)- v_2(\bar y,\bar w)-\beta ( \langle \bar  x\rangle^N+\langle \bar  z\rangle^N+\langle \bar  y\rangle^N+\langle \bar  w\rangle^N) \\
 &\le C\omega(c\sqrt{\varepsilon})+c\,\omega^{1/2}_2(c\sqrt{\varepsilon})+\omega_L(c\sqrt{\varepsilon}+c\,\omega_2(c\sqrt{\varepsilon})),
\end{align*}
which leads to a contradiction as soon as we make $\varepsilon>0$ small enough. 
\end{proof}

This comparison principle is interesting in itself.  For instance,   it implies that the lipschitzianity conditions on the current cost under which the optimality principles in \cite{Sor99,GS04,M04,MS15} were obtained, can be replaced by mere continuity plus growth assumptions.
\vsm
 
\noindent{\it Proof of Proposition \ref{Th_supsol}}  From \cite[Thm. 8.1]{CL94}  it follows immediately that, given a MRF $W$ for some $p_0$ and $\gamma$,   the decrease condition \eqref{MRH}
is equivalent to the viscosity supersolution condition 
\begin{equation}\label{supsol}
 \max_{u\in U}\left\{-\langle D^-W(z),f(z,u)\rangle -p_0(W(z))l(z,u)-\gamma(W(z)) \right\}\ge 0 
\end{equation}
for all $z\in\R^n\setminus\C$.  For any  $0<b<c$, we define the sets 
$$
\S_{b}:=\{z\in\overline{\R^n\setminus\T}\mid \ \ W(z)< b\}, \quad \S_{(b,c)}:=\{z\in\overline{\R^n\setminus\T}\mid \  \ b<  W(z)< c\}.
$$
Since $W$ is continuous,  proper, and positive definite and $\partial\T$ is compact, these sets are open,  bounded, and nonempty.
\vsm
{\it Step 1.} Fix $M>0$. Then, the  set  $\S:=\overline{\S}_{M+1}$ is contained in the interior of $\S':=\overline{\S}_{M+2}$ and we can define a $C^1$ function $\tilde\eta:\R^n\to[0,1]$ such that
$$
\tilde\eta(z):=
\begin{cases} 
1 \qquad \text{if} \ z\in\S, \\
0 \qquad \text{if} \ z\in\R^n\setminus\S'.
\end{cases}
$$
For all $(z,u)\in\R^n\times U$,  we set 
$$
\tilde f(z,u):=\tilde\eta(z)f(z,u), \quad \tilde \ell(z,u):=\tilde\eta(z)\left[p_0(W(z))l(z,u)+\gamma(W(z))\right]\ (\ge0).
$$
The  functions $\tilde f$, $\tilde \ell$ are continuous, bounded, $x\mapsto \tilde f(x,u)$ is (globally) Lipschitz continuous  and $x\mapsto \tilde  \ell(x,u)$ is uniformly continuous,  uniformly w.r.t. the control. Since $\tilde\eta\ge0$, from \eqref{supsol} it follows that $W$ also satisfies  
\begin{equation}\label{supsol1}
 \max_{u\in U}\left\{-\langle D^-W(z),\tilde f(z,u)\rangle -\tilde \ell(z,u) \right\}\ge 0 \quad\text{for all } z\in\R^n\setminus\T.
\end{equation}

{\it Step 2.} Fix $\varepsilon\in(0,1)$  and let $\eta_\varepsilon:\R^n\to[0,1]$ be a $C^1$  function such that
$$
 \eta_\varepsilon(z):=
\begin{cases} 
1 \qquad \text{if} \ z\in\S_{\left(\varepsilon,\frac{1}{\varepsilon}\right)}, \\
0 \qquad \text{if} \ z\in\R^n\setminus \S_{\left(\frac{\varepsilon}{2},\frac{2}{\varepsilon}\right)}.
\end{cases}
$$ 
Setting,   for all $(z,u)\in\R^n\times U$, 
$$
\tilde f_\varepsilon(z,u):=\eta_\varepsilon(z)\tilde f(z,u), \quad \tilde \ell_\varepsilon(z,u):=\eta_\varepsilon(z)\tilde \ell(z,u),
 $$
we finally obtain that $W$  satisfies
\begin{equation}\label{supsol2}
 \max_{u\in U}\left\{-\langle D^-W(z),\tilde f_\varepsilon(z,u)\rangle -\tilde\ell_\varepsilon(z,u) \right\}\ge 0 \quad\text{for all } z\in\R^n.
\end{equation}
Let $\lambda:\R\to(0,1)$ be a $C^1$  function such that $0<\dot\lambda\le \bar C$ for some $\bar C>0$, $\lambda(s)\to 0$ as $s\to-\infty$ and $\lambda(s)\to 1$ as $s\to+\infty$, as, for instance,  $\lambda(s)=\frac{1}{\pi}\left(\arctan(s)+\frac{\pi}{2}\right)$. Hence, by \cite[Prop.\, 2.5]{BCD},  the function
$$
V(z,r):=\lambda(W(z)+r) 
$$
 turns out to be  a bounded, continuous, and nonnegative viscosity supersolution of the Hamilton-Jacobi-Bellman equation
$$ 
\max_{u\in U}\left\{-\langle Dv(z,r)\,,\,\tilde F_\varepsilon(z,u)\rangle \right\}=0 \quad \text{for all }(z,r)\in\R^{n+1},
$$
where $\tilde F_\varepsilon(z,u):=(\tilde f_\varepsilon(z,u),\tilde \ell_\varepsilon(z,u))$. Set $\Psi(z,r):=V(z,r)$. Since $V\ge0$, given $\sigma>0$,     $V$ is also a viscosity supersolution of the  obstacle equation
\begin{equation}\label{supsol3}
\min\left\{ \sigma\,v(z,r)+\max_{u\in U}\left\{-\langle Dv(z,r)\,,\,\tilde F_\varepsilon(z,u)\rangle \right\}, v(z,r)-\Psi(z,r)\right\}=0 \  \text{on} \ \R^{n+1}.
\end{equation}
By introducing a new control $a\in\{0,1\}$ and defining $\hat F_\varepsilon(z,u,a):=a\, \tilde F_\varepsilon(z,u)$,   $\hat L(z,r,u,a):=(1-a)\sigma\Psi(z,r)$,  we can  reformulate \eqref{supsol3} as  
\begin{equation}\label{supsol4}
 \sigma\,v(z,r)+\min_{a\in\{0,1\}}\max_{u\in U}\left\{-\langle Dv(z,r)\,,\,\hat F_\varepsilon(z,u,a)\rangle -\hat L(z,r,u,a)\right\}=0 \  \text{on} \ \R^{n+1},
\end{equation}
where all the assumptions of  Lemma \ref{LCP} are met. Thus, \eqref{supsol4} (equivalently,  \eqref{supsol3}) satisfies the comparison principle for bounded viscosity sub- and supersolutions. In particular, by a standard dynamic programming procedure, the unique bounded, continuous solution  of \eqref{supsol3} is the value  function
$$
V^\varepsilon_\sigma(z,r):=\inf_{u\in\M([0,+\infty),U)}\sup_{T\ge0}e^{-\sigma T}\Psi(\tilde x_\varepsilon(T),\tilde x^0_\varepsilon(T))  \qquad\text{for all} \  (z,r)\in\R^{n+1},
$$
where   $(\tilde x_\varepsilon,\tilde x^0_\varepsilon)$  is the unique solution of  the control system $(\dot  x,\dot  x^0)=\tilde F_\varepsilon(x,u)$ with initial condition $(z,r)$. From the comparison principle, 
\begin{equation}\label{est_sigma}
V(z,r)\ge V^\varepsilon_\sigma(z,r)  \qquad\text{for all} \  (z,r)\in\R^{n+1}.
\end{equation}
Given $z\in\S_{\left(\varepsilon,\frac{1}{\varepsilon}\right)}$ and for every $u\in\M([0,+\infty),U)$,  set 
$$
\tilde T_z^\varepsilon(u):=\inf \left\{t\ge 0\mid \ \  W(\tilde x_\varepsilon(t))\ge \frac{1}{\varepsilon} \ \text{or} \  W(\tilde x_\varepsilon(t))\le \varepsilon \right\}\le+\infty.
$$ 
Clearly, $\tilde T_z^\varepsilon(u)>0$,  as  $\varepsilon<W(z)<\frac{1}{\varepsilon}$. Furthermore,  for all $t\in[0,\tilde T_z^\varepsilon(u))$,   the  solution $(\tilde x_\varepsilon,\tilde x^0_\varepsilon)$ corresponding to $u$ coincides with the solution, say  $(\tilde x,\tilde x^0)$, of  the control system $(\dot  x,\dot  x^0)=(\tilde f,\tilde\ell)(x,u)$ with  $(\tilde x,\tilde x^0)(0)=(z,r)$.  Hence, letting $\sigma$ tend to zero in \eqref{est_sigma}, for every $(z,r)\in\S_{\left(\varepsilon,\frac{1}{\varepsilon}\right)}\times \R$,  we get the inequality
$$
V(z,r)\ge \inf_{u\in\M([0,+\infty),U)}\sup_{T\in\left[0,\tilde T_z^\varepsilon(u)\right)} V(\tilde x(T),\tilde x^0(T)),
$$
that, by a recursive procedure as in \cite{Sor99}, finally implies  
\begin{equation}\label{est_0}
V(z,r)\ge \inf_{u\in\M([0,+\infty),U)}\sup_{T\in\left[0,\tilde T_z(u)\right)} V(\tilde x(T),\tilde x^0(T)) \quad\text{for all} \ (z,r)\in\R^{n+1},
\end{equation}
if $\tilde T_z(u)$ is the first time at which $\tilde x$ reaches the target $\C$ (possibly equal to $+\infty$).
\vsm
{\it Step 3.} Let $z\in\overline{\S}_M\setminus\C$, namely $0<W(z)\le M$, and fix   $0<\rho\le -\lambda(M)+\lambda(M+1)$. From \eqref{est_0} with initial point $(z,0)$,  it follows that there exists a control $u\in\M([0,+\infty),U)$ such that, for any $T\in[0, \tilde T_z(u))$, we get
$$
\displaystyle V(\tilde x(T),\tilde x^0(T)) =\lambda\left(W(\tilde x(T))+\int_0^{T}\tilde\ell(\tilde x(t),u(t))\,dt\right)\le V(z,0)+\rho\le \lambda(M+1),
$$
namely,
$$
\displaystyle W(\tilde x(T))+\int_0^{T}\tilde\ell(\tilde x(t),u(t))\,dt \le M+1.
$$
Therefore,   $\tilde x(t)$ belongs to $\S$ for all $t\in[0,\tilde T_z(u))$, so that $\tilde\eta$ as in Step 1 is identically 1. As a consequence, we have    $\tilde f(x,u)\equiv f(x,u)$, $\tilde\ell(x,u)\equiv p_0(W(x))l(x,u)+\gamma(W(x))$, $\tilde x(\cdot)\equiv x(\cdot\,,u,z)$, and  $\tilde T_z(u)\equiv T_z(u)$ as in Def. \ref{Admgen}, so that, in particular $u\in\U(z)$. By the arbitrariness of $M>0$,  the above considerations imply that $W$ satisfies \eqref{SOP} for all $z\in\R^n\setminus\C$. The proof of Proposition \ref{Th_supsol} is thus complete. \qed

\bibliographystyle{plain}
 \bibliography{lyap}




%

\end{document}